\documentclass[leqno, 11pt]{scrartcl}
\usepackage{hyperref}
\usepackage[left=2cm,right=2cm,top=1.2cm,bottom=1.1cm,includeheadfoot]{geometry}
\usepackage[utf8x]{inputenc}
\usepackage[english]{babel}
\usepackage{amssymb,amsmath,amstext,amsthm,amsfonts}

\usepackage{fancyhdr}
\newtheorem{thm}{Theorem}[section]
\newtheorem{cor}[thm]{Corollary}
\newtheorem{lem}[thm]{Lemma}
\newtheorem{prop}[thm]{Proposition}
\newtheorem{assump}[thm]{Assumption}

\theoremstyle{definition}
\newtheorem{defi}[thm]{Definition}
\newtheorem{rem}[thm]{Remark}

\numberwithin{equation}{section}

\newcommand{\la}{\langle}
\newcommand{\ra}{\rangle}
\newcommand{\dom}{\mathrm{dom}}

\newcommand{\cl}{\mathcal}

\newcommand{\cd}{\cdot}
\newcommand{\R}{\mathbb{R}}
\newcommand{\N}{\mathbb{N}}
\newcommand{\C}{\mathbb{C}}
\newcommand{\slim}{\underset{n\rightarrow\infty}{\rm{s- lim}}~}

\DeclareMathOperator*{\esssup}{ess\,sup}

\newcommand\dN{{\mathbb{N}}}

\newcommand{\ga}{{\alpha}}
\newcommand{\gb}{{\beta}}

\newcommand{\gD}{{\Delta}}
\newcommand{\gga}{{\gamma}}

\newcommand{\gL}{{\Lambda}}

\newcommand{\go}{{\omega}}

\newcommand{\gs}{{\sigma}}

\newcommand\gt{{\tau}}

\newcommand\cA{{\mathcal{A}}}
\newcommand\cB{{\mathcal{B}}}
\newcommand\cC{{\mathcal{C}}}

\newcommand\cG{{\mathcal{G}}}

\newcommand\cI{{\mathcal{I}}}

\newcommand\cK{{\mathcal{K}}}

\newcommand\cU{{\mathcal{U}}}

\title{Convergence rate estimates for Trotter product approximations of solution operators for non-autonomous Cauchy problems }
\author{Hagen \textsc{Neidhardt}\footnote{H. Neidhardt: WIAS Berlin, Mohrenstr. 39, D-10117 Berlin, Germany;
email: hagen.neidhardt@wias-berlin.de}, 
Artur \textsc{Stephan}\footnote{A. Stephan: HU Berlin, Institut f\"ur Mathematik, Unter den Linden 6, D-10099 Berlin, Germany;
email: stephan@math.hu-berlin.de}, and Valentin A. \textsc{Zagrebnov}\footnote{V.A.Zagrebnov: Universit\'{e} d'Aix-Marseille - Institut de Math\'{e}matiques de Marseille  (UMR 7373), CMI - Technop\^{o}le Ch\^{a}teau-Gombert, 39, rue F. Joliot Curie, 13453 Marseille, France, email: valentin.zagrebnov@univ-amu.fr}}

\begin{document}
\maketitle

\begin{abstract}
In the present paper we advocate the Howland-Evans approach
to solution of the abstract non-autonomous Cauchy problem (non-ACP) in a separable Banach space $X$.
The main idea is to reformulate this problem as an autonomous Cauchy problem (ACP) in a new  Banach space
$ L^p(\cI,X)$, $p \in [1,\infty)$, consisting of $X$-valued functions on the time-interval $\cI$.
The fundamental observation is a one-to-one correspondence between solution operators (propagators)
for a non-ACP and the corresponding evolution semigroups for ACP in $L^p(\cI,X)$.
We show that the latter also allows to apply a full power of the operator-theoretical methods to scrutinise
the non-ACP including the proof of the Trotter product approximation formulae with operator-norm estimate
of the rate of convergence. The paper extends and improves some recent results in this direction in particular
for Hilbert spaces.
\end{abstract}
\bigskip\bigskip\bigskip
\thanks{
Primary 34G10, 47D06, 34K30; Secondary 47A55.}

\newpage
\tableofcontents 

\section{Introduction} \label{sec:1}
The theory of evolution equations plays an important role in various areas of pure and applied mathematics,
physics and other natural sciences. Since the early 1950s, starting with papers by T.Kato \cite{Kato1953} and
R.S.Phillips \cite{Phillips1953}, research in this field became very active and it still enjoys a lot of
attention. A comprehensive introduction to this topic is presented in \cite[Chapter VI. 9.]{EngNag2000} and
also in the book by W.Tanabe \cite{Tan1979}.

A general Cauchy problem for linear non-autonomous evolution equations in a Banach space has the form
\begin{align}\label{CauchyProblem}
 \dot u(t) = -C(t)u(t), ~~ u(s)=u_s \in X, ~~ 0 <s\leq t\leq T,
\end{align}
where $\{C(t)\}_{t\in \cI}$ is a one-parameter (time-dependent) family of closed linear operators in the
separable Banach space $X$. Here the time-interval $\cI:=[0,T]\subset \R$ and we also introduce $\cI_0 : =(0,T]$.
To solve the non-autonomous Cauchy problem (non-ACP) \eqref{CauchyProblem} means to find a so-called
\textit{solution operator} (or \textit{propagator}): $\{U(t,s)\}_{(t,s)\in \Delta}$,
$\Delta=\{(t,s)\in \cI_0\times \cI_0: 0<s\leq t\leq T\}$, with the property that $u(t)=U(t,s)u_s$,
$(t,s)\in \Delta$, is in a certain sense a solution of the problem \eqref{CauchyProblem} for an appropriate set of initial data $u_s$.

By definition, propagator $\{U(t ,s)\}_{(t,s)\in \Delta}$ is supposed to be a strongly continuous
operator-valued function $U(\cd ,\cd):\Delta\rightarrow \cl{B}(X)$ satisfying the properties:
\begin{align*}\label{FunctionalEquation}
&U(t,t)=I \quad \mbox{for} \quad t\in \cI_0 \ ,\\
&U(t,r)U(r,s)=U(t,s) \quad \mbox{for} \quad t,r,s\in \cI_0 \quad \mbox{with~} \quad s\leq r\leq t \ ,\\
&\|U\|_{\cB(X)} :=\sup_{(t,s)\in\Delta}\|U(t,s)\|_{\cl{B}(X)}<\infty \ .
\end{align*}
For details see Definition \ref{SolutionDefinition} in \S\ref{sec:3.1}.

We note that there are essentially two different approaches to solve the abstract linear non-ACP
\eqref{CauchyProblem} in the normed vector spaces.

The first method consists of approximation of the operator family $\{C(t)\}_{t\in \cI}$ by operators
$\{\{C_n(t)\}_{t\in \cI}\}_{n\in \N}$, for which the corresponding Cauchy problem
\begin{align*}
 \dot u(t) = -C_n(t)u(t), ~~ u(s)=u_s \in X, ~~ 0<s\leq t\leq T
\end{align*}
can be easily solved. Often, the family of operators $\{C(t)\}_{t\in \cI}$ is approximated by a piecewise
constant operators, see T.Kato \cite{Kato1970, Kato1973}.
Then one encounters the problem: In which sense the sequence of approximating propagators
$\{\{U_n(t,s)\}_{(t,s)\in \Delta}\}_{n\in\N}$ converges to the solution operator $\{U(t,s)\}_{(t,s)\in \Delta}$
of the non-ACP \eqref{CauchyProblem} ?

Another approach allows to solve the problem \eqref{CauchyProblem} using perturbation, or extension,
theory for linear operators. It does not need any approximation scheme, see for example
\cite{EngNag2000,Nei1981,NeiZag2009}. This approach is quite flexible and can be used in
very general settings. Its main idea can be described as follows:

The non-ACP in $X$ can be reformulated as an \textit{autonomous} Cauchy problem
(ACP) in a new  Banach space $ L^p(\cI,X)$, $p \in [1,\infty)$, of $p$-summable functions on the
time-interval $\cI$ with values in the Banach space $X$.

In the second approach a central notion is the \textit{evolution generator} $\cl K$ on $ L^p(\cI,X)$.
It generates a semigroup $\{\cU(\gt) = e^{-\gt \cK}\}_{\gt\ge0}$ on $L^p(\cI,X)$
which is called an \textit{evolution semigroup}. In turn the evolution semigroup on $L^p(\cI,X)$
is entirely defined by propagator $\{U(t,s)\}_{(t,s)\in\gD}$ in such a way that the representation
\begin{equation}\label{CorrespondenceEvolutionSemigroupPropagator}
\begin{split}
 (e^{-\tau \cl K}f)(t)=(\cl U(\tau)f)(t)=
 \begin{cases}
  U(t,t-\tau)f(t-\tau), &{\rm ~~if~} t\in (\tau, T] \ , \\
  0, &{\rm ~~if~} 0\leq t\leq \tau \ ,
 \end{cases}
\end{split}
\end{equation}
holds for any $f\in  L^p(\cI,X)$. In the following we use the short notation
\begin{align*}
 (e^{-\tau \cl K}f)(t)=(\cl U(\tau)f)(t)=U(t,t-\tau)\chi_\cI(t-\tau)f(t-\tau) \ .
\end{align*}
It turns out that there is a one-to-one correspondence between the set of all evolution generators and the
set of all propagators. Moreover, the important observation is that the set of all evolution generators in
$L^p(\cI,X)$ can be characterised quite independently from propagators,
see Theorem \ref{PropagatorEvolutionGeneratorCorrespondeceLpSetting}.

Notice that in this paper we use a definition of the generator of a semigroup which differs from the
usual one by the sign, see \eqref{CorrespondenceEvolutionSemigroupPropagator}.
It turns out that this choice of definition is more convenient for our presentation.

The first problem we have to solve is: How to find the evolution generator for a non-ACP~\eqref{CauchyProblem} ?
To this aim we introduce the so-called \textit{evolution pre-generator}
\begin{align*}
 \widetilde{\cl K} = D_0 + \cl C, \quad \dom(\widetilde{\cl K}) = \dom(D_0)\cap \dom(\cl C)\subset
 L^p(\cI,X) \ ,
\end{align*}
where $D_0$ is the generator of the right-shift semigroup and $\cl C$ is the multiplication operator
induced by the operator family $\{C(t)\}_{t\in \cI}$ in $L^p(\cI,X)$. Appropriate assumptions
on the family $\{C(t)\}_{t\in \cI}$ guarantee that operator $\cl C$ is a generator in $ L^p(\cI,X)$.
If $\{U(t,s)\}_{(t,s)\in\Delta}$ is the solution operator of the non-ACP
\eqref{CauchyProblem}, then it turns out that the generator $\cK$ of the associated evolution semigroup
$\{\cU(\gt)\}_{\gt \ge 0}$ defined by \eqref{CorrespondenceEvolutionSemigroupPropagator} is a closed operator
extension of the evolution pre-generator $\widetilde{\cl K}$.
Conversely, if the evolution pre-generator $\widetilde{\cK}$ admits a closed extension, which is
an evolution generator, then the corresponding propagator $\{U(t,s\}_{[t,s)\in\gD}$ can be regarded as a
solution operator of the non-ACP \eqref{CauchyProblem}.

In general, it is difficult to answer the question: whether an evolution pre-generator admits a closed  extension,
which is an evolution generator ? However, the problem simplifies if the pre-generator is closable and its closure
is a generator. In this case one gets that the closure is already an \textit{evolution generator}.
Then obviously the evolution pre-generator admits only \textit{one} extension, which is a generator and, hence,
which is an evolution generator. This means, that the non-ACP \eqref{CauchyProblem} is solvable and even
uniquely.

From the point of view of the operator theory the problem formulated above fits into the question: whether
the sum of two generators is \textit{in essential} a generator, i.e. whether the closure of the sum of two
generators is a generator.

If the sum of two generators $A$ and $B$ of contractions semigroups in some Banach space is in essential a
generator, then the so-called  \textit{Trotter product formula}
\begin{align*}
 e^{-\tau C} = \slim (e^{-\tau A/n}e^{-\tau B/n})^n, \quad C:= \overline{A+B} \ ,
\end{align*}
in the \textit{strong} operator topology, is valid for the closure $\overline{A+B}$.
This formula goes back to Sophus Lie (1875) for bounded linear operators. Later it was generalised by
H.Trotter to unbounded generators of contraction semigroups, see \cite{Trotter1959}.
The formula admits a further generalisation to an arbitrary pair $\{A,B\}$ of generators of semigroups
if their semigroups satisfy a so-called Trotter \textit{stability} condition, cf. Proposition
\ref{StabilityProposition}.

Note that generalisation \cite{Trotter1959} says nothing about the convergence-rate of the Trotter product
formula and by consequence about the error-bound for approximation by this formula the solution operators. 
To this aim one has to consider the convergence of the Trotter product formula in the 
operator-norm topology. For the case of Banach spaces see \cite{CachZag2001}. However, in \cite{CachZag2001} 
the operator $A$ was assumed to be generator of a holomorphic semigroup. In our case, this assumption is not 
satisfied for a principal reason: the evolution semigroup (\ref{CorrespondenceEvolutionSemigroupPropagator}) 
can never be a holomorphic semigroup! Nevertheless, some observations of \cite{CachZag2001} admit a 
generalisation to the case of evolution semigroups. We discuss this point in  Remark \ref{rem:7.9}.

Finally, after having determined the convergence rate of the Trotter product formula in the operator-norm,
one has to carry over this result to the propagator approximations. It turns out that the Trotter product
formula yields an approximation of the propagator $\{U(t,s)\}_{(t,s)\in \gD}$ in the operator norm, which has
(uniformly in $\gD$) the \textit{same} convergence-rate as the Trotter product formula for the evolution
semigroup, see Theorems  \ref{TheoremEstimate} and \ref{EstimatePropagators}.

We express a hope that these results might be useful in applications since they give a uniform error estimate
for a discretized approximation of the solution operator  $\{U(t,s)\}_{(t,s)\in \gD}$ for the non-ACP
\eqref{CauchyProblem}, see \S \ref{sec:7}. In particular, it concerns the numerical simulations, where
as a palliative approach one uses some domain-dependent error estimates for operator \textit{splitting}
schemes in the strong operator topology \cite{Batkai2011}.

Now we give an overview of the contents of the paper in more details.
Our aim is analysis a linear non-ACP of the form
\begin{align}\label{EvolutionProblem}
 \dot u(t) = - Au(t) - B(t)u(t), ~~ u(s)=u_s \in X, ~~ 0<s\leq t\leq T \ ,
\end{align}
where $A$ is a generator of a bounded holomorphic semigroup and $\{B(t)\}_{t\in \cI}$ is a family of the closed
(for any time-interval $\cI=[0,T]$) linear operators in $X$. To proceed we make the following assumptions:
%
\begin{assump}\label{ass:1.1}
{\rm
Let $\ga \in (0,1)$ and $X$ be a (separable) Banach space.

 \item[\;\;(A1)] The operator $A$ is a generator of a bounded holomorphic semigroup of class $\cl G(M_A,0)$
 (\cite{Kato1980}, Ch.IX, \S1.4) with zero in the resolvent set: $0\in\varrho(A)$.

 \item[\;\;(A2)] The operators $\{B(t)\}_{t\in \cI}$ are densely defined and closed for a.e.
$t\in \cI$ and it holds that $\dom(A)\subset\dom(B(t))$ for a.e. $t\in \cI$. Moreover, for all $x\in \dom(A)$
the function $t\mapsto B(t)x$ is strongly measurable.

 \item[\;\;(A3)] For a.e. $t\in \cI$ and some $\ga \in (0,1)$ we demand that $\dom(A^\alpha)\subset\dom(B(t))$ and that
 \begin{align*}
  C_\alpha:=\mathrm{ess ~sup}_{t\in \cI}\|B(t)A^{-\alpha}\|_{\cl B(X)}<\infty \ .
 \end{align*}

\item[\;\;(A4)] Let $\{B(t)\}_{t\in\cI}$ be a family of generators in $X$ that for all $t\in \cI$
belong to the same class $\cl G(M_B,\beta)$. The function $\cI\ni t\mapsto (B(t)+\xi)^{-1}x\in X$ is
strongly measurable for any $x\in X$ and any $\xi>\beta$.

 \item [\;\;(A5)] We assume that $\dom(A^*)\subset \dom(B(t)^*)$ and
 \begin{align*}
 C_1^*:={\rm ~ess~sup}_{t\in \cI}\|B(t)^*(A^*)^{-1}\|_{\cl B(X^*)}<\infty,
 \end{align*}
 where $A^*$ and  $B(t)^*$ denote operators which are adjoint of $A$ and $B(t)$, respectively.

 \item[\;\;(A6)] There exists $\gb \in (\ga,1)$ and a constant $L_\gb > 0$ such that
 for a.e. $t,s\in \cI$ one has the estimate:
 \begin{align*}
  \|A^{-1}(B(t)-B(s))A^{-\alpha}\|_{\cl B(X)}\leq L_\beta|t-s|^\beta \ .
 \end{align*}
}
\end{assump}

We comment here that, the assumptions (A4) and (A3) imply assumption (A2). So, assuming (A4) and (A3) we
can drop the assumption (A2).

Let $\cl A$ and $\cB$ be the multiplication operators \textit{induced} by $A$ and
$\{B(t)\}_{t\in \cI}$ in $L^p(\cI,X)$. Further let $D_0$ be the generator of the right-shift semigroup in
$L^p(\cI,X)$. Note that since $A$ is a semigroup generator in $X$, the operator $\cl A$ in the space $L^p(\cI,X)$
is also a generator. Moreover, the semigroup $\{e^{-\tau\cl A}\}_{\tau\geq0}$ commutes with the semigroup
$\{e^{-\tau D_0}\}_{\tau\geq0}$. Therefore, the product $\{e^{-\tau\cl A}e^{-\tau D_0}\}_{\tau\geq0}$ defines a
semigroup and its generator is denoted  by $\cl K_0$. Note that $\cl K_0 = \overline{D_0 + \cA}$ , i.e.
a closure of the operator sum $D_0 + \cA$. In general, domain of the generator $\cl K_0$ can be
larger than $\dom(\cl A)\cap\dom(D_0)$. A widely used assumption about the operator $A$ is its
\textit{maximal parabolic regularity}, see \cite{Acquistapace1987, Prato1984, PruessSchnaubelt2001, Arendt2007}.
This means that the operator sum $D_0+\cl A$ is already closed, i.e. $\cK_0 = D_0 + \cA$ and $\dom(\cl K_0)=
\dom(\cl A)\cap \dom(D_0)$. In this paper the maximal parabolic regularity is \textit{not} supposed
for our purposes.

Now our first of the main results can be formulated as follows:
%
\begin{thm}\label{TheoremKbecomesGeneratorIntroduction}
Let the assumptions (A1), (A2) and (A3) be satisfied.
Then, the operator $\cl K :=\cl K_0+\cl B$ with $\dom(\cl K)=\dom(\cl K_0)\cap \dom(\cl B)$,
is an evolution generator in $L^p(\cI,X)$, $p \in [1,\infty)$, and  the non-autonomous Cauchy  problem
\eqref{EvolutionProblem}
has a unique solution operator $\{U(t,s)\}_{(t,s)\in \gD}$ in the sense of Definition \ref{SolutionDefinition}.
\end{thm}
%
The proof of this theorem mainly uses a perturbation theory due to J.Voigt \cite{Voigt1977},
see Proposition \ref{VoigtsTheorem}. Note that the theorem holds without assuming that operators
$\{B(t)\}_{t\in\cI}$ are generators.

We comment that if the assumption (A4) is satisfied, then the induced multiplication operator $\cl B$
becomes a generator that also belongs to $\cG(M,\gb)$. A pair $\{\cK_0,\cB\}$ is called
\textit{Trotter-stable} if it satisfies the condition
\begin{align*}
 \sup_{n\in \dN}\sup_{\gt \ge 0}\left\|\left(e^{-\tau\cK_0/n}e^{-\tau\cB /n}\right)^n\right\| < \infty \ .
\end{align*}
If the pair $\{\cK_0,\cB_0\}$ is Trotter-stable,  then by the Trotter product formula the
evolution semigroup
$\{e^{-\gt \cK}\}_{\gt \ge 0}$ admits the representation
\begin{align}\label{eq:1.4}
 e^{-\tau \cl K} = \slim (e^{-\tau\cB /n}e^{-\tau\cK_0/n})^n.
\end{align}
It turns out that the pair $\{\cB,\cK_0\}$ is Trotter-stable if the operator family $\{B(t)\}_{t\in\cl I}$ is
$A$-stable (see Definition \ref{StabilityDefinition}). Let us mention here that the pair $\{\cK_0,\cB\}$ is
Trotter-stable if and only if the pair
$\{\cB,\cK_0\}$ is Trotter-stable, i.e. the estimate:
\begin{align*}
 \sup_{n\in \dN}\sup_{\gt \ge 0}\left\|\left(e^{-\tau\cB/n}e^{-\tau\cK_0 /n}\right)^n\right\| < \infty \ ,
\end{align*}
holds. In particular this yields that one can interchange operators $\cK_0$ and $\cB$ in formula
\eqref{eq:1.4}. Note that Trotter stability condition is always satisfied for generators of contraction
semigroups.

Let $\{U(t,s)\}_{(t,s)\in \Delta}$ be the propagator corresponding to the evolution semigroup
$\{e^{-\tau \cl K}\}_{\tau\geq0}$ via (\ref{CorrespondenceEvolutionSemigroupPropagator}). Then the
Trotter product formula yields an approximation of the propagator $\{U(t,s)\}_{(t,s)\in \Delta}$ in
the strong operator topology and we prove in this paper the following assertion:
%
\begin{thm}\label{StrongConvergencePropagatorIntroduction}
Let the assumptions (A1), (A3) and (A4) be satisfied. If the family $\{B(t)\}_{t\in\cI}$ is $A$-stable
(see Definition \ref{StabilityDefinition}), then
\begin{equation}\label{eq:1.6a}
\lim_{n\to\infty}\sup_{\tau\in\cI}\int_0^{T-\tau}\|\{U_n(s+\tau,s)-U(s+\tau,s)\}x\|^p_Xds = 0, \quad x \in X,
\end{equation}
for any $p\in[1,\infty)$, where the Trotter product approximation $\{\{U_n(t,s)\}_{(t,s)\in \Delta}\}_{n\in\N}$
is defined by
\begin{equation}\label{ApproximationPropagatorIntroduction}
\begin{split}
 U_n(t,s)  &:= \prod_{j=1}^{n\leftarrow} G_j(t,s\,;n), \quad n = 1,2,\ldots\:,\\
G_j(t,s\,;n)   &:= e^{-\frac{t-s} n B(s+ j\frac{t-s} n)}e^{-\frac {t-s} n A},\quad j = 0,1,2,\ldots,n,
\end{split}
\end{equation}
$(t,s) \in \gD$, with the increasingly ordered product in $j$ from the right to the left.
\end{thm}

Our second main result shows that the convergence in \eqref{eq:1.6a} can be improved from the \textit{strong}
to the \textit{operator-norm} topology and that the \textit{convergence-rate} can be estimated from above.
%
\begin{thm}
Let the assumptions  (A1), (A3), (A4), (A5), and (A6) be satisfied.
If the family of generators $\{B(t)\}_{t\in\cI}$ is $A$-stable and $\gb \in (\ga,1)$,  then there is a constant
$C_{\alpha, \beta}>0$ such that
 \begin{equation}\label{eq:1.8}
 \esssup_{(t,s)\in\gD}\|U_n(t,s)-U(t,s)\|_{\cl B(X)} \le \frac{C_{\ga,\gb}}{n^{\gb-\ga}}, \quad
 n = 2,3,\ldots \ .
\end{equation}
\end{thm}

Now few remarks are in order. Recall that in \cite{IchinoseTamura1998} Ichinose and Tamura proved under stronger
assumptions a sharper than (\ref{eq:1.8}) convergence rate in the \textit{operator norm}.
Namely, they showed that the it is of the order $O({\ln(n)}/{n})$ if the both $A$ and $B(t)$
are positive self-adjoint operators in a Hilbert space and $\{B(t)\}_{t\geq 0}$ are Kato-infinitesimally-small
with respect to $A$.
On the other hand, in \cite{Batkai2011, Batkai2012} B\'atkai \textit{et al} investigated approximations
of solution
operators for non-autonomous evolution equations by a different type of so-called
\textit{operator splittings} in the \textit{strong} operator topology.
They include, as particular, symmetrised/nonsymmetrised time-dependent Trotter product approximations in
the strong operator topology studied by \cite{VWZ2008, VWZ2009}, as well as some other Trotter-Kato product
formulae, see e.g. \cite{NeiZag1998}. In the first paper \cite{Batkai2011}, the authors proved the strong
operator
convergence and established for the non-autonomous parabolic case an optimal \textit{domain-dependent}
convergence rate for the (sequential) splitting approximation. The second paper \cite{Batkai2012} is
devoted to a detailed analysis of the case of bounded perturbations.

Equation \eqref{EvolutionProblem} describes various problems related to the linear non-ACP.
As an example, we consider in \S \ref{sec:8} the diffusion equation perturbed by a time-dependent
$t \mapsto V(t,\cdot)$ scalar potential:
\begin{align}\label{EvolProbLaplaceAndTimeDependentPotentialIntroduction}
 \dot u(t) = \Delta u(t) - V(t,x)u(t), ~~ u(s)=u_s \in L^q(\Omega), ~~ 0<s\leq t\leq T, ~x\in \Omega,
\end{align}
where $\Omega\subset\R^d$ is a bounded domain with $C^2$- boundaries and $q\in(1,\infty)$. Let
\begin{align*}
 V(t,x):\cI\times\Omega\rightarrow \C, ~~ {\rm Re}(V(t,x))\geq0 {\rm ~~for~t\in \cI, ~a.e.}~x\in\Omega.
\end{align*}
be a measurable scalar time-dependent potential. Assuming regularity of the potential $V(t,x)$,
the conditions (A3), (A5)  and (A6) can be  satisfied. As an example for the case of $d=3$, we have the
following theorem.
%
\begin{thm}\label{ExampleUniqueSolutionIntroduction}
Let $\Omega\subset \R^3$ be a bounded domain with $C^2-$boundary. Let $\alpha\in(0,
{1}/{2})$ and $q\in (3, {3}/{2\alpha})$. Choose $\varrho\in[{3}/{(2\alpha)},\infty]$,
$\beta\in(\alpha,1)$ and $\tau\in[{3}/{(2\alpha+2)},\infty]$. Let $B(t)f=V(t,\cdot)f$ define a
scalar-valued multiplication operator in $X=L^q(\Omega)$
with $V\in L^\infty(\cI, L^{\varrho}(\Omega))\cap C^\beta(\cI, L^\tau(\Omega))$ and ${\rm Re}(V(t,x))\geq0$.
Then, the evolution problem
(\ref{EvolProbLaplaceAndTimeDependentPotentialIntroduction}) has a unique solution operator
$\{U(t,s)\}_{(t,s) \in \gD}$, which admits the approximation
 \begin{align*}
  {\rm sup}_{(t,s)\in\Delta}\|U_n(t,s)-U(t,s)\|_{\cl B(L^q(\Omega))} = O(n^{-(\beta-\alpha)}),
 \end{align*}
 where for $(t,s) \in \gD$ the approximating propagator $U_n(t,s)$ is defined by the product formula
 \begin{equation*}
\begin{split}
  U_n(t,s) &:= \prod^{n\leftarrow}_{j=1} G_j(t,s),\quad n = 1,2,\ldots\; ,\\  
	G_j(t,s\,;n) &:= e^{-\frac{t-s}{n}V(s + j\frac{t-s}{n},\cdot)}e^{\frac{t-s}{n}\Delta}, \quad
j = 1,2,\ldots,n \ .
	\end{split}
 \end{equation*}
\end{thm}

The conditions for other values of the parameters when $d\geq2$, $q\in(1,\infty)$, are formulated in
\S\ref{sec:8}.

This paper is organised as follows. In \S\ref{sec:2} we summarise some basic facts about the semigroup
theory, the fractional powers of operators and the multiplication operators.
In \S\ref{sec:3}, we describe our approach to solution of the non-ACPs.
In \S\ref{sec:4} the existence of unique solution operator for our case of the linear non-ACP is proved.
\S\ref{sec:5} presents the basic properties of stability, whereas \S\ref{sec:6} investigates convergence of
the Trotter-type product approximations in the strong topology. \S\ref{sec:7} contains the proof of the
lifting of these convergence to the operator-norm topology. An application to a nonstationary diffusion
equation is the subject of \S\ref{sec:8}. Appendix (\S\ref{sec:9}) completes the presentation by some
important auxiliary statements and formulae.

Finally we point out that the paper is partially based on the master thesis \cite{Stephan2016} of one of
the authors. There a special case was treated when involved semigroups are contractions.
This allows to avoid stability considerations formulated in \S\ref{sec:5}.
In addition, in \cite{Stephan2016} a similar approach was also developed for the space $C_0(\cI,X)$
instead of $L^p(\cI,X)$. The $C_0(\cI,X)$-approach allows to prove the results similar to that
for $L^p(\cI,X)$, however, under stronger regularity assumptions on the family $\{B(t)\}_{t\in\cI}$.

\section{Recall from the theory of semigroups} \label{sec:2}

Below we recall some basic facts from the operator and semigroup theory, which are indispensable
for our presentation below.

Throughout this paper we are dealing with a separable Banach space denoted by $(X,\|\cd\|_{X})$.
Let $S$ and $T$ be two operators in $X$. If $\dom(S)\subset\dom(T)$
and there are constants $a,b\geq0$ such that
 \begin{align*}
  \|Tx\|_X\leq a\|Sx\|_X+b\|x\|_X, ~~x\in\dom(S),
 \end{align*}
 then the operator $T$ is called $S$-bounded with the relative bound $a$.

We define the resolvent of operator $A$ by $R(\lambda,A)=(A-\lambda)^{-1}:X\rightarrow \dom(A)$
when $\lambda$ is from the resolvent set $\varrho(A)$.
A family $\{T(t)\}_{t\geq0}$ of bounded linear operators on the
Banach space $X$ is called a strongly continuous (one-parameter) semigroup if it satisfies the
functional equation
\begin{align*}
    T(0)=I, ~~T(t+s)=T(t)T(s), ~~t,s\geq0,
\end{align*}
and the orbit maps $[0,\infty) \ni t\mapsto T(t)x$ are continuous for every $x \in X$. In the following we
simply call them semigroups.

For a given semigroup its generator is a linear operator defined by the limit
 \begin{align*}
  Ax:=\lim_{h\searrow0}\frac{1}{h}(x - T(h)x)
 \end{align*}
on domain
 \begin{align*}
  \dom(A):=\{x\in X: \lim_{h\searrow0}\frac{1}{h}(x - T(h)x)~~ \mathrm{exists}\}.
 \end{align*}
Note that in our definition of the semigroup \textit{generator} differs from the \textit{standard} one by the
sign minus, cf. \cite{Kato1980}.

It is well-known that the generator of a strongly continuous semigroup is a closed and densely defined
linear operator, which uniquely determines the semigroup (see e.g. \cite[Theorem
I.1.4]{EngNag2000}). For a given generator $A$ we will write $\{T(t)=e^{- t A}\}_{t \ge 0}$,
for the corresponding semigroup.

Recall that for any semigroup $\{T(t)\}_{t\geq0}$ there are constants $M_A,\gamma_A$, such that it holds
$\|T(t)\|\leq M_A e^{\gamma_A t}$ for all $t\geq0$. These semigroups are known as \textit{quasi-bounded}
of class $\mathcal{G}(M_A,\gamma_A)$ and following the Kato book we write that $A\in\mathcal{G}(M_A,\gamma_A)$
for its generator \cite[Ch.IX]{Kato1980}. If $\gamma_A\leq0$, $\{T(t)\}_{t\geq0}$ is called a \textit{bounded}
semigroup. For any semigroup we can construct a bounded semigroup by adding some constant $\nu\geq\gamma_A$ to
its generator. Then the operator $\tilde{A} := A+\nu$ generates a bounded semigroup $\{\tilde T(t)\}_{t\geq0}$
with $\|\tilde T(t)\|\leq M_A$. If $\|T(t)\|\leq1$, the semigroup is called a \textit{contraction} semigroup
and correspondingly a \textit{quasi-contraction} semigroup, if the property $\|T(t)\|\leq e^{\gamma_At}$ holds.

It is known (see \cite[Ch.IX]{Kato1980}) that for a generator $A\in\cl{G}(M_A,\gamma_A)$ the open half
plane $\{z\in\mathbb{C}:{\rm Re}(z) < - \gamma_A\}$ is contained in the resolvent set $\varrho(A)$
of $A$ and one has the estimate $\|R(\lambda, A)^k\|\leq {M_A}/({- \rm Re}(\lambda)-\gamma_A)^k$ for
the resolvent $R(\lambda, A)= (A - \lambda)^{-1}$ and the natural $k \in \mathbb{N}$.
Note that if $A\in\cl{G}(M_A,\gamma_A)$, then $\tilde{A} = A + \nu \in \cl{G}(M_A, \gamma_A -\nu)$.
Therefore, the open half-plane $\{z\in\mathbb{C}: {\rm Re}(z) < \nu -\gamma_A\}$ is contained in the
resolvent set $\varrho(\tilde A)$.

Note that the semigroup $\{T(t)\}_{t\geq 0}$ on $X$ is called a \textit{bounded holomorphic}
semigroup if its generator $A$ satisfies: $\mathrm{ran}(T(t))\subset\dom(A)$ for all $t>0$, and
$\sup_{t>0}\|tAT(t)\|\leq M <\infty$. Recall, that in this case the bounded semigroup $\{T(t)\}_{t\geq 0}$
has a unique analytic continuation into the open sector
$\{z\in\mathbb{C}:|\arg(z)|<\delta(M)\} \subset\mathbb{C}$ of the angle $\delta(M) >0$, which is
a monotonously decreasing function of $M$ such that $\lim_{M\rightarrow \infty}\delta(M) = 0$, see e.g.
\cite[Ch.1.5]{Zag2003}.

For a short recall from the perturbation theory of semigroups see \S\ref{sec:4}

\subsection{Fractional powers} \label{sec:2.1}

We recall here some facts about the fractional powers of linear operators, see e.g.
\cite[Chapter 2.6]{Pazy1983}. To this end  assume that $A$ is a generator of a bounded holomorphic semigroup
$\{e^{-t A}\}_{t\geq0}$ and $0\in\varrho(A)$. Then the fractional power for $\alpha\in(0,1)$
is defined by
\begin{align*}
 A^{-\alpha}=\frac{1}{\Gamma(\alpha)}\int_0^\infty t^{\alpha-1}
 e^{-tA}dt \ ,
\end{align*}
where $\Gamma: \mathbb{R}_{+} \rightarrow \mathbb{R}$ is the Bernoulli gamma-function.
Moreover, we define $A^{0}=I$. Thus, the operator family $\{A^{-\alpha}\}_{\alpha\geq0}$ defines a semigroup
of bounded linear operators and the operators $A^{-\alpha}$ for $\alpha>0$ are invertible
\cite[2.6.5-6]{Pazy1983}.
So, for $\alpha\geq0$ we can define $A^\alpha:=(A^{-\alpha})^{-1}$.  With this definition, we get
$\dom(A^\alpha)\subset\dom(A^\beta)$ for $\alpha\geq\beta>0$. In particular, we have
$\dom(A)\subset\dom(A^\alpha)$ for every $\alpha\in(0,1)$.

The following facts are also well-known.
\begin{prop}\label{PropertiesFractionalPowers}
Let $A$ be generator of a bounded holomorphic semigroup.
\begin{enumerate}
 \item [\;\;\rm (i)]Then there is a constant $C_0$ such that for all $\mu>0$ it holds
 \begin{align*}
  \|A^{\alpha}(A+\mu)^{-1}\|\leq C_0\mu^{\alpha-1}.
 \end{align*}
 \item [\;\;\rm (ii)]For $\mu>0$ and  $0<\alpha<1$ it holds that
 \begin{align*}
  \dom((A+\mu)^\alpha)=\dom(A^\alpha)
 \end{align*}

\end{enumerate}

\end{prop}

One of the basic tool for analysis of bounded holomorphic evolution semigroups is summarised by the
following proposition.
%
\begin{prop}[{\cite[Theorem 2.6.13]{Pazy1983}}]\label{HolomorphicSemigroupEstimate}
 Let $A$ be generator of a bounded ho\-lomorphic semigroup $U(z)$ and
 $0\in\varrho(A)$. Then for $0<\alpha$, we get
 \begin{align*}
  \sup_{t>0}\|t^\alpha A^\alpha U(t)\|=M^A_\alpha<\infty.
 \end{align*}
\end{prop}


\subsection{Multiplication operators} \label{sec:2.2}

Let $\cI=[0,T]$ be a compact interval. We consider the Banach spaces $L^p(\cI,X)$, $p \in [1,\infty)$, of
Lebesgue $p$-summable $X$-valued functions.
The dual space $L^p(\cI,X)^*$ of $L^p(\cI,X)$ is defined by the sesquilinear duality relation
$\la \cdot , \cdot \ra $, which generates bounded functionals:
\begin{equation*}
L^p(\cI,X) \times L^p(\cI,X)^* \ni (f, \Gamma) \mapsto \la f , \Gamma \ra  \in \mathbb{C} \ .
\end{equation*}
Then the following statement characterises the space $L^p(\cI,X)^*$.
\begin{prop}[{\cite[Theorem 1.5.4]{CembranosMendoza1997}}]\label{CharaterizationLpdual}
Let $1\leq p<\infty$, and let $p'$ be defined by $ (p')^{-1}  + (p)^{-1}=1$. Then for each
$\Gamma\in L^p(\cI,X)^*$ there exists a function $\Psi:\cI\rightarrow X^*$ such that {\rm {:}}
\item[\;\;\rm (i)] $\Psi$ is $w^*$-measurable, i.e. measurable are the functions
$t\mapsto \la f(t),\Psi(t)\ra$ for all $f\in L^p(\cI,X)$ ,

\item[\;\;\rm (ii)] The function $\|\Psi(\cdot)\|_{X^*}:t\mapsto \|\Psi(t)\|_{X^*}$ is measurable and
belongs to $L^{p'}(\cI)$,

\item[\;\;\rm (iii)] If $\la f , \Gamma \ra = \int_\cI dt \la f(t),\Psi(t) \ra $ for
    all $f\in L^p(\cI,X)$, then $\|\Gamma\|=\|~\|\Psi(\cdot)\|_{X^*} ~\|_{L^{p'}}$.

Conversely, each $w^*$-measurable function $\Psi:\cI\rightarrow
X^*$, for which there is $g\in L^{p'}(\cI)$ such that
$\|\Psi(t)\|_{X^*}\leq g(t)$ for a.e. $t\in \cI$, induces by {\rm (iii)}
a continuous linear functional $\Gamma$ on $L^p(\cI,X)$, whose norm
is less than or equal to $\|g\|_{L^{p'}}$.
\end{prop}

An important role plays in the following the  so-called \textit{multiplication} operators on
the Banach space $L^p(\cI,X)$, $p \in [1,\infty)$. A function $\phi\in L^{\infty}(\cI)$ defines a
multiplication operator $M(\phi)$ on $L^p(\cI,X)$  by
 \begin{align*}
  (M(\phi)f)(t):=\phi(t)f(t) ~~{\rm for ~a.e.~} t\in \cI, ~~\dom(M(\phi))=L^p(\cI,X).
 \end{align*}
Moreover, let $C(\cI,X)$ be the Banach space of all continuous functions
$f:\cI \longrightarrow X$ endowed with the supremum norm. By
$C_0(\cI, X)$ we denote the subspace of $C(\cI,X)$ of all continuous functions, which vanish at $t = 0$.
%
\begin{defi}\label{DefinitionDenseCrossSection}
 We say the set $\cl D\subset L^p(\cI,X)$ has a dense \textit{cross-section} in $X$ if
  \item[\;\;\rm (i)] $\cl D\subset L^p(\cI,X)\cap C(\cI,X)$ ,

  \item[\;\;\rm (ii)] for any $t\in\cI_0$ the set $
   [\cl D]_t := \{x\in X: \exists f\in\cl D {\rm ~such~ that~} \hat f
   (t)=x\}$ is dense in $X$ , where $\hat{f}$ denotes the unique \textit{continuous} representative
   of $f \in \cl D$.
\end{defi}

Using definition of the multiplication operator $M(\phi)$ and the cross-section density
property we find a condition when a linear set is dense in $L^p(\cI,X)$. Let us denote by
$W^{k,p}(\cI)$, $k \in \dN$, $p \in [1,\infty]$ the Sobolev space over $\cI$. Then one gets
the following statement.
%
\begin{prop}\label{DensityDenseCrossSectionLp}
 If a linear set $\cl D\subset L^p(\cI,X)$ has a dense cross-section
 in $X$ and if for every $\phi\in W^{1,\infty}(\cI)$ one has {\rm{:}}
$M(\phi)\cl D\subset \cl D$, then $\cl D$ is dense in $L^p(\cI,X)$.
\end{prop}
%
\begin{proof}
Let $\Gamma \in L^p(\cI,X)^*$ be a functional on $L^p(\cI,X)$ such
that
 \begin{align*}
  \la f, \Gamma\ra=0, ~~f\in\cl D.
 \end{align*}
Now we use the characterisation of the dual space $L^p(\cI,X)^*$ given by Proposition
\ref{CharaterizationLpdual}. Then there is a $w^*$-measurable function $\Psi:\cI\rightarrow X^*$ such that
\begin{align*}
 0=\la f, \Gamma \ra =\int_\cI \la f(t), \Psi(t)\ra~ dt, {\rm ~for~} f \in\cl D.
\end{align*}
By virtue of $M(\phi)\cl D\subset\cl D$ for $\phi\in W^{1,\infty}(\cI)$, it follows that
\begin{align*}
 0=\int_\cI \la \phi(t)f(t), \Psi(t)\ra~ dt=\int_\cI \phi(t)\la f(t), \Psi(t)\ra~ dt \ ,
\end{align*}
i.e., the function $t\mapsto \la f(t), \Psi(t)\ra$ is in $L^1(\cI)$. Since
$\phi\in W^{1,\infty}(\cI)$ is arbitrary, we conclude that
\begin{align*}
 0=\la f(t), \Psi(t)\ra, \rm{~for~ a.e.~}t\in \cI \ .
\end{align*}
Then the condition that $\cl D$ has a dense cross-section implies $\Psi(t)=0$ for
a.e. $t\in \cI$ and hence $\Gamma=0$.
\end{proof}

Now, let $\{C(t)\}_{t\in \cI}$ be a family of linear operators in the Banach space $X$.
We note that the domains of operators $\{C(t)\}_{t\in \cI}$ may depend on the parameter $t$.
The multiplication operator $\cl C$ in $L^p(\cI, X)$ induced by $\{C(t)\}_{t\in \cI}$ is defined by
\begin{equation}\label{Mult-OperatorLp}
\begin{split}
 (\cl{C}f)(t)&:=C(t)f(t) \ , \ \ {\rm{with \ domain}} \\
 \dom(\cl{C})&:=\left\{
f\in L^p(\cI,X)~:
\begin{matrix}
&f(t)\in\dom(C(t)) {\rm ~for~ a.e.~} t\in \cI\\
&\cI \ni t\mapsto C(t)f(t)\in L^p(\cI,X)
\end{matrix}
\right\}.
\end{split}
\end{equation}
%
\begin{prop}\label{ClosenesInducedOperatorLp}
 If $\{C(t)\}_{t\in \cI}$ is a family of closed linear operators in
 $X$, then the induced operator $\cl{C}$ is also closed.
\end{prop}
%
\begin{proof}
Let the sequence $\{f_n\}_{n\geq 1}\subset \dom(\cl{C})$  be such that limits:
$\lim_{n\rightarrow\infty} f_n = f$ and $\lim_{n\rightarrow\infty} {\cl C} f_n = g$,
exist in the $L^p(\cI, X)$-topology. This implies that by a diagonal procedure one can find a subsequence
$\{f_{n_k}\}_{k\geq 1}$ such that $\lim_{{n_k}\rightarrow\infty}  f_{n_k}(t) = f(t)$ and
$\lim_{{n_k}\rightarrow\infty} ({\cl C} f_{n_k})(t) = g(t)$ for a.e. $t\in \cI$.
Since for any $t\in \cI$ the operator $C(t)$ is closed in $X$, we conclude that
$f(t)\in\dom(C(t))$ and $ C(t) f(t) = g(t)$ for almost all $t\in \cI$. On the other hand, since $g\in
L^p(\cI, X)$, it follows that $f\in\dom(\cl C)$ and that $(\cl C f)(t) = g (t)$ almost
everywhere in $\cI$. The latter proves that operator $\cl C$ is closed.
\end{proof}

For a family of generators, we have the following theorem.
%
\begin{thm}\label{InducedGeneratorLp}
 Let $\{C(t)\}_{t\in \cI}$ be a family of generators in $X$ such
 that for almost all $t\in \cI$ it holds that $C(t)\in \cl
 G(M,\beta)$ for some $M\geq 1$ and $\beta\in \mathbb{R}$.
 If the function $\cI\ni t\mapsto (C(t)+\xi)^{-1}x\in X$ is strongly
 measurable for $\xi > \beta$, $x\in X$, then the induced
 multiplication operator $\mathcal{C}$ is a generator in $L^p(\cI, X)$, $p \in [1,\infty)$,
 and the corresponding semigroup $\{e^{- \tau \cl{C}}\}_{\tau \geq 0}$ is given by
 \begin{align*}
  (e^{-\tau \cl{C}}f)(t)=e^{-\tau C(t)}f(t) {\rm ~~for~ a.e.~} t\in \cI.
 \end{align*}
In particular, on obtains that $\cC \in \cG(M,\gb)$.
\end{thm}
%
\begin{proof}
Let $\cl J\subset\cI$ be a Borel set with characteristic function $\chi_\cl J(\cdot)$.
For $\xi>\beta$ and $x\in X$ we define the mappping
\begin{align*}
f_{\cl J,\xi}:= \xi(C(\cdot)+\xi)^{-1}\chi_\cl J(\cdot)x:\cI\rightarrow X.
\end{align*}
Then, by definition $f_{\cl J,\xi}$ is element of $L^p(\cI,X)$ , $f_{\cl J,\xi}(t)\in\dom(C(t))$ for a.e.
$t\in \cI$ and
\begin{align*}
 C(t)f_{\cl J,\xi}(t)= \xi \, \chi_\cl J(t)\, x - \xi^2 \, (C(t)+\xi)^{-1}\chi_\cl J(t)\, x
\end{align*}
is also an element of $L^p(\cI,X)$. Hence, $f_{\cl J,\xi}\in\dom(\cl C)$. Since for a.e. $t\in \cI$
the operator $C(t)$ is a generator in $X$ the Yosida approximation argument yields that
\begin{align*}
 f_{\cl J,\xi}(t)\rightarrow \chi_\cl J(t)x, {\rm~for~}\xi\rightarrow \infty, ~x\in X, {\rm ~a.e.~}t\in \cI.
\end{align*}
Note that it is valid for any $\cl J\subset\cI$. Therefore, $\dom(\cl C)$ is dense in $L^p(\cI,X)$.

Now, we estimate the iterated resolvents. Recall that for any $t\in \cI$ the operators $C(t)$ belong to
the same class $\cl G(M,\beta)$. Thus, for any $k\in \mathbb{N}$ we have
\begin{align*}
 \|(C(t)+\lambda)^{-k}\|_{\cl B(X)}\leq \frac{M}{(\lambda-\beta)^k}, ~~\lambda>\beta.
\end{align*}
Hence, for almost every $t\in \cI$ and any $f\in L^p(\cI,X)$, we obtain
\begin{align*}
 \|((\cl C+\lambda)^{-k}f)(t)\|_X=\|(C(t)+\lambda)^{-k}(f(t))\|_X\leq
 \frac{M}{(\lambda-\beta)^k}\|f(t)\|_X, ~~\lambda>\beta.
\end{align*}
This implies that
\begin{align*}
  \|(\cl C+\lambda)^{-k}f\|_{L^p}\leq\frac{M}{|\lambda-\beta|^k}\|f\|_{L^p}, ~~\lambda>\beta,
\end{align*}
and therefore, by the Hille-Yosida Theorem (see e.g. \cite[Theorem 2.3.8]{EngNag2000}) it follows that
$\cl C$ is a generator in $L^p(\cI,X)$. The corresponding semigroup is given by the Euler limit:
\begin{align*}
 e^{-\tau\cl C}f = \lim_{n\rightarrow \infty} \left(I+\frac{\tau}{n}\cl C \right)^{-n}f \ ,
\ f \in L^p(\cI,X) \ .
\end{align*}
For any $n\geq0$, we have
\begin{align*}
 \left(\left(I+\frac{\tau}{n}\cl C \right)^{-n}f\right)(t) = \left(I+\frac{\tau}{n}C(t) \right)^{-n}f(t) \ .
\end{align*}
This yields
\begin{align*}
 (e^{-\tau\cl C}f)(t) &= \lim_{n\rightarrow \infty}\left(\left(I+\frac{\tau}{n}\cl C \right)^{-n}f\right)(t)
 = \lim_{n\rightarrow \infty} \left(I+\frac{\tau}{n}C(t) \right)^{-n}f(t) = e^{-\tau C(t)}f(t) \ ,
\end{align*}
which coincides with expression claimed in theorem.
\end{proof}
%
\begin{rem}
 We note that the domain of the generator $\cl C$ does not necessarily have a dense cross-section in $X$
 since its elements might be not continuous.
\end{rem}

An operator $A$ in $X$, that does not depend on the time-parameter $t$, trivially induces a multiplication
operator $\cl A$ in $L^p(\cI, X)$ given by
\begin{align*}
 (\cl{A}f)(t):=Af(t) ~~{\rm for~ a.e.~}t\in \cI
\end{align*}
with
\begin{align*}
 \dom(\cl{A}):=
\left\{f\in L^p(\cI, X):
\begin{matrix}
&f(t)\in\dom(A) ~{\rm for~ a.e.~}t\in \cI\\
&\cI \ni t\mapsto Af(t)\in L^p(\cI, X)
\end{matrix}
\right\}.
\end{align*}
Then Theorem \ref{InducedGeneratorLp} immediately yields the following corollary:
%
\begin{cor}\label{InducedConstantGeneratorLp}
 Let $A$ be a generator in $X$. Then the induced multiplication operator $\mathcal{A}$ is a generator
 in $L^p(\cI,X)$ and its semigroup is given by
 \begin{align*}
  (e^{-\tau \cl{A}}f)(t)=e^{-\tau A}f(t), {\rm ~~a.e.~}t\in \cI.
 \end{align*}
\end{cor}

The next lemma describes how domains of {two} induced multiplication operators in $L^p(\cI, X)$
can be described by domains of the corresponding operators in the space $X$.
\begin{lem}\label{DomainsInclusionsLp}
Let the assumptions (A1) and (A2) be satisfied. If for each $x \in \dom(A)$
\begin{align}\label{a:2.2}
\esssup_{t\in\cI}\|B(t)x\|_X\leq C_0 \|Ax\|_X<\infty \quad \mbox{for} \quad x\in \dom(A) \ ,
\end{align}
is valid, then $\dom(\cl A)\subset\dom(\cl B)$.
\end{lem}
%
\begin{proof}
Let $f\in\dom(\cl A)$. Then, by definition of $\dom(\cl A)$ one gets $f(t)\in \dom(A)$ for a.e.
$t\in \cI$ and hence $f(t)\in \dom(B(t))$ for a.e. $t\in \cI$. Consequently, by virtue of \eqref{a:2.2}
we obtain
\begin{align*}
 \esssup_{t\in\cI}\|B(t)A^{-1}\|_{\cl B(X)}\leq C_0 \ .
\end{align*}
Hence, one gets
\begin{align*}
 \|B(t)f(t)\|_X = \|B(t)A^{-1}Af(t)\|_X \leq C_0\|Af(t)\|_X \ ,
\end{align*}
which yields that the function $t\mapsto B(t)f(t)$ is in $L^p(\cI, X)$. Thus, $f\in\dom(\cl B)$,
i.e. $\dom(\cl A)\subset\dom(\cl B)$.
\end{proof}

Note that a family $\{F(t)\}_{t\in\cI}$ of bounded operators is measurable if the map
$\cI \ni t \mapsto F(t)x \in X$ is measurable for each $x \in X$. The following
proposition is very useful for our purposes.
%
\begin{prop}[{\cite{Evans1976}}]\label{OperatorNormsBoundedOperator}
Let $\{F(t)\}_{t\in \cI}$ be a measurable family of bounded linear operators on $X$.
Then, for the induced multiplication operator $\cl F$ on $L^p(\cI,X)$ its norm
can be expressed as
 \begin{align*}
  \|\cl F\|_{\cl B(L^p(\cI,X))}= \esssup_{t\in\cI}\|F(t)\|_{\cl B(X)} \ .
 \end{align*}
\end{prop}

\section{Non-autonomous Cauchy problems and the evolution semigroups approach to solve them} \label{sec:3}

Let us consider the non-ACP \eqref{CauchyProblem} in the separable Banach space $X$. We are going to explain
an approach of solving it by using the \textit{evolution semigroups}.

\subsection{Evolution semigroup approach} \label{sec:3.1}

Crucial for this approach is the notion of the \textit{evolution pre-generator}.
%
\begin{defi}\label{evol-pre-gener}
An operator $\cl K$ in $L^p(\cI,X)$, $p \in [1,\infty)$, is called a evolution pre-generator if
 \item[\;\;\rm (i)] $\dom(\cl K)\subset C(\cI,X)$ and $M(\phi)\dom(\cl K)\subset\dom(\cl K)$ for
 $\phi\in W^{1,\infty}(\cI)$,

 \item[\;\;\rm (ii)] $\cl KM(\phi)f-M(\phi)\cl Kf=M(\dot{\phi})f, ~~f\in\dom(\cl K),
 ~~\phi\in W^{1,\infty}(\cI)$, where $\dot \phi=\partial_t\phi$,

 \item[\;\;\rm (iii)] the domain $\dom(\cl K)$ has a dense cross-section in $X$ (see Definition
 \ref{DefinitionDenseCrossSection}).\\[-2ex]

\noindent
If, in addition, the operator $\cl K$ is a generator of a semigroup in $L^p(\cI,X)$,
then $\cl K$ is called an \textit{evolution generator}.
\end{defi}
\begin{rem}
 The domain $\dom(\cl K)$ of an evolution pre-generator is dense in the Banach space $L^p(\cI,X)$.
Indeed, the dense cross-section property {\rm (iii)} together with {\rm (i)} and
Lemma  \ref{DensityDenseCrossSectionLp} imply the density of $\dom(\cl K)\subset L^p(\cI,X)$.
\end{rem}

Now, we can present the main idea concerning the solving of the problem \eqref{CauchyProblem}. The next
theorem explains why we are interested in such a notion as \textit{evolution semigroups}.
%
\begin{thm}[{\cite[Theorem 4.12]{Nei1981}}]\label{PropagatorEvolutionGeneratorCorrespondeceLpSetting}
 Between the set of all semigroups $\{e^{-\tau \cl K}\}_{\tau\geq0}$ on the Banach space $L^p(\cI,X)\}$,
 $p \in [1,\infty)$, generated by an evolution generator $\cl K$ and the set of all solution operators (propagators)
 $\{U(t,s)\}_{(t,s)\in\Delta}$ on the Banach space $X$ exists a one-to-one correspondence such that
 the relation
 \begin{align}\label{CorrespondenceEvolutionSemigroupPropagatorFormula}
  (e^{-\tau \cl K}f)(t)=U(t,t-\tau)\chi_\cI(t-\tau)f(t-\tau),
 \end{align}
  holds for $f\in L^p(\cI, X)$ and for a.e. $t\in \cI$.
\end{thm}

In other words, there is a \textit{one-to-one} correspondence between evolution semigroups and the propagators
that solve the non-ACP problem \eqref{CauchyProblem}

One of the important example of evolution generator is $D_0:=\partial_t$ defined in the space $L^p(\cI,X)$
by
\begin{align*}
 D_0f(t):=\partial_t f(t), ~ \dom(D_0):=\{f\in W^{1,p}([0,T],X): f(0)=0\} \ .
\end{align*}
Then, the operator $D_0$ is a generator of class $\mathcal{G}(1,0)$ of the \textit{right-shift} evolution
semigroup $\{S(\tau)\}_{\tau\geq0}$ that has the form
\begin{align*}
 (e^{-\tau D_0}f)(t)=(S(\tau) f)(t):=f(t-\tau)\chi_\cI(t-\tau), ~~f\in L^p(\cI,X), {\rm ~~a.e.~} t\in \cI.
\end{align*}
The propagator corresponding to the right-shift evolution semigroup is the {identity propagator},
i.e. $U(t,s) = I$ for $(t,s) \in \Delta \in \cI_0\times \cI_0$, where $\cI_0 = \cI \setminus \{0\}$.

We note that the generator $D_0$ has empty spectrum since the semigroup $\{S(\tau)\}_{\tau\geq0}$ is
\textit{nilpotent} and therefore the integral $\int_0^\infty d\tau \, e^{-\tau\lambda}S(\tau)f$ exists for
any $\lambda\in \C$ and for any $f\in L^p(\cI,X)$.

For a given operator family $\{C(t)\}_{t\in\cI}$ in $X$ the induced multiplication operator
$\cl C$ in $L^p(\cI,X)$ is defined by (\ref{Mult-OperatorLp}). We consider in $L^p(\cI,X)$  the operator
\begin{align}\label{DefinitionKTilde}
 \widetilde {\cl K} := D_0+\cl C \, , ~~\dom(\widetilde {\cl K}) := \dom(D_0)\cap\dom(\cl C) \ .
\end{align}
%
\begin{lem}\label{DenseCrossSectionEvolutionOperatorC}
 If $\dom(\widetilde{\cl K})$ has a dense cross-section, then the operator $\widetilde{\cl K}$ is a
 evolution pre-generator.
\end{lem}
%
\begin{proof}
By (\ref{DefinitionKTilde}) we get $\dom(\widetilde {\cl K})\subset\dom(D_0)\subset C(\cI,X)$. Since $\cl C$
is an induced multiplication operator, then by definition (\ref{Mult-OperatorLp})
it commutes with the operator $M(\phi)$ for $\phi\in W^{1,\infty}(\cI)$.
So, with $\dom(\widetilde{\cl K})=\dom(D_0)\cap\dom(\cl C)$ we get $M(\phi)\dom(\widetilde {\cl K})
\subset \dom(\widetilde{\cl  K})$.
Then the relation $\widetilde{\cl K}M(\phi)f-M(\phi)\widetilde{\cl K}f=M(\dot{\phi})f$ for
$f\in\dom(\widetilde{\cl K})$ (see Definition \ref{evol-pre-gener}, {\rm (ii)}) follows by the Leibniz rule
for $(D_0M(\phi)f)(t)=\partial_t (\phi f)(t)$.
\end{proof}
Now, we precise the notion of the \textit{solution operator} of the problem \eqref{CauchyProblem} versus
the \textit{propagator} $\{U(t,s)\}_{(t,s)\in\Delta}$ on the Banach space $X$ that we first described in
Introduction \S\ref{sec:1}.

\begin{defi}\label{SolutionDefinition}~
\item[\;\;\rm (i)] The evolution non-ACP \eqref{CauchyProblem} is called \textit{correctly posed} in
$\cI_0 = \cI \setminus \{0\}$  if $\widetilde{\cl K}$ defined by (\ref{DefinitionKTilde}) is an evolution
pre-generator.
	
\item[\;\;\rm (ii)] A propagator $\{U(t,s)\}_{(t,s)\in\Delta}$ is called a \textit{solution operator} of
  the correctly posed evolution problem \eqref{CauchyProblem} if the corresponding evolution generator
  $\cl K$ (Theorem \ref{PropagatorEvolutionGeneratorCorrespondeceLpSetting}) is an operator extension of
  $\widetilde{\cl K}$, i.e. $\widetilde{\cl K} \subseteq \cl K$.
	
\item[\;\;\rm (iii)] The evolution problem \eqref{CauchyProblem} has a unique solution operator if
  $\widetilde {\cl K}$ admits only one extension that is an evolution generator.
\end{defi}
\begin{rem}~
\item[\;\;(i)] It is an open problem whether an evolution pre-generator admits several extension which are evolition generators. However, if this is case then the non-ACP \eqref{CauchyProblem} has more than one solution operator.

\item[\;\;(ii)]
Our Definition \ref{SolutionDefinition} (ii) of a correctly posed non-ACP is a \textit{weak} property.
For example, the notion of well-posedness developed in \cite{Nickel1996} implies this property. %
\end{rem}

To find extensions of the evolution pre-generator $\widetilde{\cl K}$ which are evolution generators is, in
general, a nontrivial problem. However, there is a special case, that easily guarantees the existence of such
extension and, moreover, it is unique.
%
\begin{thm}\label{EvolutionProblemUniqueSolution}
Assume that the non-ACP \eqref{CauchyProblem} is correctly posed in $\cI_0$. If the
evolution pre-generator $\widetilde {\cl K}$ is closable in $L^p(\cI,X)$ and its closure $\cl K$ is a
generator, then the evolution problem \eqref{CauchyProblem} has a unique solution operator.
\end{thm}
%
\begin{proof}
 Assume that $\cl K$ belongs to the class $\cl G(M,\beta)$. Then by Lemma 2.16 of \cite{Nei1981} the estimate
 \begin{align*}
  \|f(t)\|_X\leq \frac{M}{(\xi-\beta)^{(p-1)/p}}\|(\cl K+\xi)f\|_{L^p}\ , ~~f\in \dom(\cl K)\, ,
 \end{align*}
holds a.e. in $\cI$ for all $\xi>\beta$. In particular, one gets for any $f\in \dom(\tilde {\cl K})$:
\begin{align*}
 \|f\|_C\leq \frac{M}{(\xi-\beta)^{(p-1)/p}}\|(\widetilde {\cl K}+\xi)f\|_{L^p} \, .
\end{align*}
Hence, we conclude for the closure $\cl K$ of $\widetilde{\cl  K}$ one has $\dom(\cl K)\subset C(\cI,X)$.

Now, we show that ${\cl K}$ is an evolution generator. Let $f\in\dom(\cl K)$. Then,
by the closeness of $\cl K$, there is a sequence $f_n\in\dom(\widetilde{\cl K})$ such that $f_n\rightarrow f$
and $\widetilde Kf_n\rightarrow \cl Kf$,  both in $L^p(\cI,X)$. Let $\phi\in W^{1,\infty}(\cI)$. Since
$\widetilde{\cl K}$ is an evolution pre-generator, Definition \ref{evol-pre-gener}, {\rm (ii)} yields
 \begin{align*}
  \widetilde {\cl K} M(\phi)f_n=M(\phi)\widetilde{\cl K} f_n+M(\dot \phi)f_n.
 \end{align*}
Note that the right-hand side converges to $M(\phi)\cl K f+M(\dot \phi)f$. Therefore,
we conclude that $M(\phi)f\in\dom(\cl K)$ and $\cl KM(\phi)f=M(\phi)\cl K f+M(\dot \phi)f$. Hence,
$\cl K$ is an evolution generator.

Now let $\cl K$ and $\cl K'$ be two different extensions of $\widetilde{\cl K}$ that are both evolution
generators. Since $\cl K$ is the closure of $\widetilde {\cl K}$ and $\cl K'$ is closed, we get
$\dom(\cl K) \subseteq \dom(\cl K')$ and the restriction: $K'\upharpoonright\dom(\cl K)=\cl K$.
Recall that $e^{-s\cl K}(\dom(\cl K))\subseteq \dom(\cl K)$, for $s\geq 0$. Then for all $f\in\dom(\cl K)$
and $0 < s < \tau$ we obtain
 \begin{align*}
 \frac{d}{ds}\{ e^{-(\tau-s)\cl K'}e^{-sK}f\}=e^{-(\tau-s)\cl K'}(\cl K'-\cl K)e^{-s\cl K}f = 0 \ .
 \end{align*}
Hence, the function $s\mapsto e^{-(\tau-s)\cl K'}e^{-s\cl K}u$ is a constant for each $u\in\dom(\cl K)$.
Thus, the semigroup generated by $\cl K'$ must be the same as the one by $\cl K$, which implies
$\cl K=\cl K'$.
\end{proof}

These considerations suggest the following strategy for solving the non-ACP: \\
To find the unique solution operator of the problem \eqref{CauchyProblem} it is sufficient to prove
that the evolution pre-generator $\widetilde {\cl K}$, defined by \eqref{DefinitionKTilde},
is an \textit{essential generator}, i.e., the closure of $\widetilde{\cl K}$ is a generator.

\subsection{A special class of evolution equations} \label{sec:3.2}

We are interested in the non-ACP of a special form. Setting $C(t) :=A+B(t)$, $t\in \cI$,
$\dom(C(t))=\dom(A)\cap\dom(B(t))$ we see that this problem fits into \eqref{CauchyProblem}.

The operator $A$ in $X$ trivially induces a multiplication operator $\mathcal{A}$ in the Banach
space $L^p(\cI,X)$. The operator family $\{B(t)\}_{t\in \cI}$ induces a multiplication operator $\cl B$.
Our aim is, to show that the closure of the evolution pre-generator
\begin{align*}
 \widetilde{\cl K} := D_0+\cl A+\cl B,~~\dom(\widetilde {\cl K}) := \dom(D_0)\cap\dom(\cl A)\cap\dom(\cl B)
\end{align*}
becomes an evolution generator under appropriate assumptions on the operator $A$ and the operator family
$\{B(t)\}_{t\in\cI}$.

Firstly, we consider the operator sum $D_0+\cl{A}$. Let $A$ be a generator in $X$ with the semigroup
$\{e^{-\tau A}\}_{\tau\geq0}$.
Then $\mathcal{A}$ is a generator in $L^p(\cI,X)$ with semigroup $\{e^{-\tau \cl A}\}_{\tau\geq0}$ given by
$(e^{-\tau \mathcal{A}}f)(t)=e^{-\tau A}f(t)$ for a.e. $t\in \cI$ (cf. Lemma \ref{InducedConstantGeneratorLp}).
Since $A$ is time-independent, the operators $\cl A$ and $D_0$ commute. Hence, the product
\begin{align*}
 e^{-\tau D_0}e^{-\tau \cl A}f=\chi_\cI(\cd-\tau)e^{-\tau\mathcal{A}}f(\cd-\tau)
\end{align*}
defines a semigroup on $ L^p(\cI,X)$. The generator of this semigroup is denoted by $\cl K_0$ and  satisfies the following properties:
%
%
\begin{lem}\label{L0Properties}
 Let $A$ be a generator in $X$ inducing the multiplication operator $\cl A$ in $L^p(\cI,X)$. Let
 $D_0$ be the generator of the right-shift semigroup on $L^p(\cI,X)$. Then, the following holds:

\item[\;\;\rm (i)] The set $\cl D:=\dom(D_0)\cap\dom(\cl A)$ is dense in $L^p(\cI,X)$ and it has a
dense cross-section in $X$. In particular, $\dom(\cl K_0)$ has a dense cross-section in $X$.
	
\item[\;\;\rm (ii)] The restriction $\cl K_0\upharpoonright\cl D =: \widetilde {\cl K}_0 = D_0 + \cl A$ and
the closure $\overline{(\widetilde{\cl K_0})} = \cl K_0$.
	
\item[\;\;\rm (iii)] $\|e^{-\tau \cl K_0}\|_{\cl B(L^p(\cI,X))}=\|e^{-\tau \mathcal{A}}\|_{\cl B(L^p(\cI,X))}$
for $\tau \in \cI $. In particular, the generators $A$, $\cl A$ and $\cl K_0$ belong to the same class
$\cl G(M,\beta)$.
\end{lem}
%
\begin{proof}
(i) Note that for any $\phi\in W^{1,\infty}(\cI)$ we have $M(\phi)\cl D\subset \cl D$. Now we prove that
$\cl D$ has a dense cross-section in $X$. To this aim, let $t_0\in \cI \setminus \{0\}$ and $x_0\in X$ be
fixed. Since $A$ is a generator in $X$, by the Yosida approximation it follows that
 \begin{align*}
  \dom(A)\ni x_\xi:=\xi(A+\xi)^{-1}x_0\rightarrow x_0, {\rm~~as~}\xi\rightarrow\infty.
 \end{align*}
Therefore, for any $\epsilon>0$ there exists $\xi>0$ such that $\|x_\xi-x_0\|_X<\epsilon$. Let
$\psi\in C^\infty(\cI)$ be such that $\psi(0)=0$ and $\psi(t_0)=1$. Then, $g$ defined by
$g(t)=\psi(t)x_\xi$ is in $\cl D$ and $\|g(t_0)-x_0\|_X<\epsilon$.

Assertion (ii) holds by definition and assertion (iii) follows immediately from the fact that
$\dom(D_0)\cap\dom(\cl A)$ is dense in $L^p(\cI,X)$ and that the operator $D_0$ belongs to the class
$\cl G(1,0)$.
\end{proof}
\begin{rem}\label{MaxParReg}
 In general, the operator $\widetilde{\cl K}_0 = D_0+\cl A$ must \textit{not} be a closed operator and the
 domain of $\widetilde{\cl K}_0$ may be \textit{larger} than $\dom(D_0)\cap\dom(\cl A)$. Let 
 \begin{align*}
 \partial_{t}u(t) = - Au(t), \quad u(0) = u_0 \ ,
 \end{align*}
be the evolution problem associated to the densely defined and closed operator $A$. Let us recall that if $A$ satisfies the condition of
\textit{maximal parabolic regularity}, see e.g.
\cite{Acquistapace1987, Prato1984, PruessSchnaubelt2001, Arendt2007}, then $A$ has to be the generator of a \textit{holomorphic} semigroup and the operator $\widetilde{\cl K}_0$ is closed. Hence,
$\widetilde{\cl K}_0= \cl K_0$. However, if $A$ is the generator of a holomorphic semigroup, then in general it does not follow
that $A$ satisfies the condition of maximal parabolic regularity. This is only true for Hilbert spaces.
\end{rem}


\section{Existence and uniqueness of the solution operator of the evolution equation}\label{sec:4}

In this section we want to find the solution operator for the non-ACP \eqref{EvolutionProblem}
in the sense of Definition \ref{SolutionDefinition}. In particular, we show that the closure
$\overline{\widetilde{\cl K}}$ of the operator $\widetilde{\cl K} = D_0+\cl A+\cl B$
is a generator (cf. Theorem \ref{EvolutionProblemUniqueSolution}). In fact, we are going to prove that
$\cl K := \cl K_0+\mathcal{B}$ is an evolution generator.

Note that since we deal with many generators, there is a need to investigate the sum of them. To this
aim we recall two results from the perturbation theory for semigroup generators.
%
\begin{prop}[{\cite[Corollary IX.2.5]{Kato1980}}]\label{KatosPerturbationResults}~
   Let $A$ be the generator of a holomorphic semigroups and let
    $B$ be $A$-bounded with relative bound zero. Then $A + B$ is also
    the generator of a holomorphic semigroup.
\end{prop}

The next result is due to J. Voigt \cite{Voigt1977}. It allows to treat perturbations with
non-zero relative bounds.
%
\begin{prop}[{\cite[Theorem 1]{Voigt1977}}]\label{VoigtsTheorem}
 Let $\{T(t)\}_{t\geq0}$ be a semigroup acting on the Banach space $X$ with generator
 $A\in\cl{G}(M_A,\gamma_A)$. Let $B$
be a densely defined linear operator in $ X$ and assume there is a dense subspace $\cl D\subset X$
such that:
\item[\;\;\rm (i)]  $\cl D\subset\dom(A)\cap\dom(B)$, $T(t)\cl D\subset \cl D$ for $t\geq0$ and for all
$x\in \cl D$ the function $t\mapsto BT(t)x$ is continuous,

\item[\;\;\rm (ii)]  There are constants $\beta_1\in(0,\infty]$ and $\beta_2\in[0,1)$ such that for all
$x\in \cl D$ it holds that
  \begin{align*}
   \int_0^{\beta_1} dt \, e^{-\gamma_At}\|BT(t)x\|\leq \beta_2\|x\| \ .
  \end{align*}
Then there exists a unique semigroup $\{S(t)\}_{t\geq0}$ and its generator $C$ is the closure of
the restriction $(A+B)\upharpoonright\cl D$, with domain $\dom(C)=\dom(A)$.
Moreover, the operator $B\upharpoonright\cl D$ is $A\upharpoonright\cl D$-bounded and can be extended
uniquely to an $A$-bounded operator $\widehat B$ with domain $\dom(\widehat B)=\dom(A)$. For this extension
one gets that $C=A+\widehat B$. In particular, if $B$ is closed, then $B$ is $A$-bounded and $C=A+B$.
Moreover, the following estimate holds
\begin{align*}
 \|S(t)\|\leq \frac{M_A}{1-\beta_2}\left(\frac{M_A}{1-\beta_2}\right)^{t/\beta_1}e^{\gamma_A t},
 \quad t \ge 0 \ .
\end{align*}
\end{prop}
\begin{lem}\label{ClBClAalphaBounded}
 Assume (A1), (A2) and (A3) for the operators $A$ and the operator family $\{B(t)\}_{t\in\cI}$. Then,
 we get $\|\cl B\cl A^{-\alpha}\|_{\cl B(L^p(\cI,X))}\leq C_\alpha$.
\end{lem}
\begin{proof}
 The claim follows directly using Lemma \ref{ClosenesInducedOperatorLp} and Lemma \ref{DomainsInclusionsLp}.
\end{proof}
\begin{prop}\label{KbecomesGenerator}
Let the assumptions  (A1), (A2) and (A3) be satisfied.
Then $\cl K= \cl K_0+\mathcal{B}$ is a generator in $L^p(\cI,X)$, $p \in [1,\infty)$, with domain
$\dom(\cl K)=\dom(\cl K_0)$.
\end{prop}
%
\begin{proof}
We want to apply Proposition \ref{VoigtsTheorem}. Let $\cl D=\dom(D_0)\cap\dom(\cl A)\cap\dom(\cl B)$.
Since $\dom(\cl A)\subset\dom(\cl B)$, we have $\cl D=\dom(D_0)\cap\dom(\cl A)$. Using Lemma \ref{L0Properties},
we conclude that $\cl D$ is a dense subspace of $L^p(\cI =[0,T],X)$, which is invariant under the semigroup
$\{e^{-\tau \cl K_0}\}_{\tau \geq0}$.

From Proposition \ref{HolomorphicSemigroupEstimate} we get that for a fixed $\alpha\in (0,1)$ and for any
$\tau\in (0, T] = \cI_0$ there exists a constant $M^A_\alpha$ (which depends only on $\alpha$) such that
$\|A^\alpha e^{-\tau A}\|\leq {M^A_\alpha}/{\tau^\alpha}$.

We prove conditions (i) and (ii) of Proposition \ref{VoigtsTheorem}. Let $f\in \cl D=\dom(D_0)\cap\dom(\cl A)
\subset C_0(\cI,X)$. Then for $\alpha\in(0,1)$ and $\tau>0$ we conclude that
\begin{align*}
  \|\mathcal{B}e^{-\tau \cl K_0}f\|_{L^p}^p
 =&\int_\cI dt \|B(t)e^{-\tau D_0}e^{-\tau A}f(t)\|_X^p 
 \leq \int_\cI dt \|B(t)A^{-\alpha}A^\alpha e^{-\tau A}\|_{\cl B(X)}^p \cdot \|f\|_{L^p}^p\leq\\
 \leq& \esssup_{t\in \cI}\|B(t)A^{-\alpha}\|_{\cl B(X)}^p\cdot\frac{(M^A_\alpha)^p}{\tau^{\alpha p}} T
 \cdot\|f\|_{L^p}^p
 \leq C_\alpha^p \frac{(M^A_\alpha)^pT}{\tau^{\alpha p}}\|f\|_{L^p}^p.
\end{align*}
Then, we get $\|\mathcal{B}e^{-\tau \cl K_0}f\|_{L^p} \leq C_\alpha {M^A_\alpha T^{1/p}}
\tau^{-\alpha} \, \|f\|_{L^p}.$ Moreover, for $f\in \cl D$ we have
\begin{align*}
 \|\cl B (e^{-\tau \cl K_0}&-I)f\|_{L^p} = \|\int_0^\tau \cl B e^{-\sigma  \cl K_0}
 \cl K_0f \|_{L^p} d\sigma = \|\int_0^\tau \cl B e^{-\sigma  \cl A} e^{-\sigma  D_0} \cl K_0f \|_{L^p} d\sigma \leq \\
 \leq & \|\cl B \cl A^{-\alpha}\|_{\cl B(L^p(\cI,X))} \int_0^\tau \|\cl A^\alpha e^{-\sigma
 \cl A}\|_{\cl B(L^p(\cI,X))}d\sigma  \|\cl K_0f \|_{L^p(\cI,X)} \leq\\
 \leq & C_\alpha M^A_\alpha \int_0^\tau \frac{1}{\sigma ^\alpha }d\sigma  \|\cl K_0f \|_{L^p(\cI,X)}
 =  \frac{C_\alpha M^A_\alpha}{1-\alpha} \tau^{1-\alpha} \|\cl K_0f \|_{L^p(\cI,X)},
\end{align*}
that yields continuity in $\tau = 0$ and hence, the function $\cI\ni \tau\mapsto
\mathcal{B}e^{-\tau \cl K_0}f\in L^p(\cI,X)$ is continuous. Moreover, we get
\begin{align*}
 \int_0^a\|\mathcal{B}e^{-\tau \cl K_0}f\|_{L^p(\cI,X)} d\tau \leq C_\alpha M^A_\alpha\|f\|_{L^p}
 \int_0^a\frac{1}{\tau^\alpha}d\tau =\frac{C_\alpha M^A_\alpha}{(1-\alpha)}a^{1-\alpha}\|f\|_{L^p(\cI,X)}.
\end{align*}
Now, take $a< \left({(1-\alpha)}/{C_\alpha M^A_\alpha}\right)^{1/(1-\alpha)}$. Hence, all conditions
of Proposition \ref{VoigtsTheorem} are satisfied. So we conclude that the operator $\cl K = \cl K_0+\mathcal{B}$
with domain $\dom(\cl K)=\dom(\cl K_0)$ is a generator.
\end{proof}

Now we can state the main theorem concerning the non-ACP \eqref{EvolutionProblem}.
%
\begin{thm}\label{OurEvolutionProblemUniqueSolution}
Let the assumptions (A1), (A2) and (A3) of Assumption \ref{ass:1.1} be satisfied.
Then the evolution problem \eqref{EvolutionProblem} has a unique solution operator in the sense of
Definition \ref{SolutionDefinition}.
\end{thm}
%
\begin{proof}
 The evolution problem is correctly posed since the set $\dom(D_0)\cap\dom(\cl A)$ has a dense cross-section
 in $X$ (cf. Lemma \ref{L0Properties}). Using Theorem \ref{EvolutionProblemUniqueSolution} and
 Proposition~\ref{VoigtsTheorem} the assertion follows.
\end{proof}
\begin{rem}~

\begin{enumerate}
\item The existence result does \textit{not} require that the operators $B(t)$ are generators.
\item The assumption $0\in\varrho(A)$ is just for simplicity. Otherwise,
the generator $A$ can be shifted by a constant $\eta>0$. Proposition \ref{PropertiesFractionalPowers} ensures
that the domain of the fractional power of $A$ does not change either.
\item The assumptions (A1), (A2), (A3) imply that for a.e. $t\in\cI$ the
operator $B(t)$ is infinitesimally small with respect to $A$. Indeed, fix $t\in\cI$, we conclude
\begin{align*}
 \dom(A+\eta)=\dom(A)\subset\dom( A^{\alpha})\subset\dom(B(t))
\end{align*}
for $\eta>0$ and so by Proposition \ref{PropertiesFractionalPowers} we have
\begin{align*}
 \|B(t)(A+\eta)^{-1}\|_{\cl B(X)}&\leq\|B(t)A^{-\alpha}\|_{\cl B(X)}\cdot\| A^{\alpha}( A+\eta)^{-1}
 \|_{\cl B(X)} \leq \frac{C_\alpha C_0}{\eta^{1-\alpha}}.
\end{align*}
And therefore for any $x\in\dom( A)\subset\dom( B(t))$, we get
\begin{align*}
 \|B(t)x\|_X\leq \frac{C_\alpha C_0}{\eta^{1-\alpha}}\cdot\|(A+\eta)x\|_X\leq C_\alpha C_0\eta^{\alpha}
 \left(\frac{1}{\eta}\|Ax\|_X+\|x\|_X \right).
\end{align*}
Since the relative bound can be chosen arbitrarily small by the large shift $\eta>0$, the perturbation
Proposition \ref{KatosPerturbationResults} yields that $A+B(t)$ is the generator of a holomorphic semigroup.
Hence, the problem \eqref{EvolutionProblem} is a parabolic evolution equation.
\end{enumerate}
\end{rem}

\section{Stability condition} \label{sec:5}

As we have already mentioned, the existence result holds even if the operators $B(t)$ are not generators.
In the following, we are going to approximate the solution using a Trotter product formula. To this end,
we have to take into account the condition (A4) from  Assumption \ref{ass:1.1}.
\begin{rem}
\item[\,\;(i)]
 In (A2) we assumed that the function $t\mapsto B(t)x$ for $x\in\dom(A)\subset\dom(B(t))$ is  strongly measurable.
 The assumption (A4) implies this property, which can be easily obtained using the Yosida approximation.
 Using (A3), (A4) and $\dom(A)\subset\dom(A^\alpha)\subset\dom(B(t))$, assumption (A2) is not needed anymore.

\item[\;\;(ii)] Using Theorem \ref{InducedGeneratorLp}, assumption (A4) implies that the induced operator $\cl B$
is a generator in $L^p(\cI,X)$.
\end{rem}
Now, let us consider the operator sums $A+B(t)$ and $\cl A+\cl B$.
\begin{lem}\label{SumClAClBContractionGenerator}
Let the operators $A$ and $\{B(t)\}_{t\in\cI}$ satisfy assumptions (A1), (A3) and (A4) and let $\cl A$ and
$\cl B$ be the corresponding induced multiplication operators in $ L^p(\cI,X)$. Then, $C(t) := A+B(t)$ is
generator of a holomorphic semigroup on $X$ and it induces the multiplication operator $\cl C$ given by
$\cl C = \cl A+\cl B$, which is in turn a generator of a holomorphic semigroup on $ L^p(\cI,X)$.
\end{lem}
%
\begin{proof}
 Using Lemma \ref{ClBClAalphaBounded} and Theorem \ref{KatosPerturbationResults}, we obtain that $C(t)$ and
 $\cl C$ generate holomorphic semigroups.
\end{proof}
A fundamental tool for approximation the solution operator (propagator) $\{U(t,s)\}_{(t,s)\in\Delta}$
of the evolution equation \eqref{EvolutionProblem} is the \textit{Trotter product formula}.
The first step is to establish a general sufficient condition for existence of this formula in the case
of evolution semigroups \S\ref{sec:3.1}.
\begin{prop}[{\cite[Theorem 3.5.8]{EngNag2000}}]\label{ClassicalTrotter}
 Let $A$ and $B$ be two generators in $X$. If there are constants $M > 0$ and $\go \in \R$ such that the
 condition
 \begin{align}\label{eq:5.1}
  \|(e^{-\tau/n A}e^{-\tau/n B})^n\|_{\cl B(X)}\leq Me^{\omega \tau} \ ,
 \end{align}
is satisfied for all $\tau\geq0$ and $n \in \N$ and if the closure of the sum: $C=\overline{A + B}$, is
in turn a generator, then the corresponding semigroup is given by the Trotter product formula
 \begin{align}\label{eq:5.2a}
  e^{-\tau C}x=\lim_{n\rightarrow \infty}(e^{-\tau/n A}e^{-\tau/n B})^n \, x,~~x\in X \ ,
 \end{align}
with uniform convergence in $\tau$ for compact intervals.
\end{prop}
\begin{rem}
The condition \eqref{eq:5.1} is called the Trotter \textit{stability condition} for the pair of operators
$\{A,B\}$. It turns out that if the Trotter stability condition is satisfied for the pair $\{A,B\}$, then
the Trotter stability condition holds also for
the pair $\{B,A\}$, i.e. there are constants $M' > 0$, and $\omega' \in \R$ such that
\begin{align*}
\|(e^{-\tau/n B}e^{-\tau/n A})^n\|_{\cl B(X)}\leq M'e^{\omega' \tau}
 \end{align*}
for all $\tau\geq 0$ and $n\in \N$. In particular, the operators $A$ and $B$ can be interchanged in formula
\eqref{eq:5.2a} without modification of the left-hand side.
%
\end{rem}
In the following, we consider two different splittings of the evolution semigroup generator $\cl K $,
see \S\ref{sec:3.2}:
\begin{align*}
 \cl K = \overline{D_0 + (\cl A + \cl B)} = \overline{D_0 + \cl C}, ~~{\rm and~~}\cl K =  \cl K_0 + \cl B \ .
\end{align*}
For them we want to apply the Trotter product formula \eqref{eq:5.2a}. Note that the Trotter stability condition
\eqref{eq:5.1} can be expressed in terms of operators $A$ and $B(t)$.
%
\begin{defi}\label{StabilityDefinition} Let $X$ be a separable Banach space.
\begin{enumerate}
 \item Let $\{C(t)\}_{t\in\cI}$ be family of generators in $X$. The family $\{C(t)\}_{t\in\cI}$ is
 called stable if there is a constant $M > 0$ such that
 \begin{align}\label{eq:5.2}
 \esssup_{(t,s) \in \gD}\left\|\prod_{j=1}^{n\leftarrow} e^{-\frac{(t-s)}{n} C(s +
 \frac j n(t-s))}\right\|_{\cl B(X)} \le M
 \end{align}
holds for any $n\in\N$.

 \item Let $A$ be a generator and let $\{B(t)\}_{t\in\cI}$ be a family of generators in $X$. The family
 $\{B(t)\}_{t\in\cI}$
is called $A$-stable if there is a constant $M > 0$ such that
 \begin{align}\label{eq:5.3}
  \esssup_{(t,s)\in \gD}\left\|\prod_{j=1}^{n\leftarrow} G_j(t,s;n)\right\|_{\cl B(X)} \le M
 \end{align}
holds for any $n\in\N$ where $G_j(t,s;n)$ is defined by \eqref{ApproximationPropagatorIntroduction}.
\end{enumerate}
In the both cases these products are ordered for index $j$ increasing from the right to the left.
\end{defi}
\begin{rem}
There are different types of stability conditions known for the evolution equations. This is,
in particular, a condition of the \textit{Kato-stability}, which is equivalent to the \textit{renormalizability}
condition for the underlying Banach space, see \cite[Definition 4.1]{NeiZag2009}. We note that below
condition of the Trotter stability involves only the products (\ref{eq:5.2}), (\ref{eq:5.3}), of valued for
equidistant-time steps $(t-s)/n$. Therefore, it is \textit{weaker} than the Kato-stability condition.
\end{rem}
\begin{prop}\label{StabilityProposition}
Let $A$ be a generator and  let $\{B(t)\}_{t\in\cI}$ be a family of generators in the separable Banach space $X$.
Let $\cl A$ and  $\cl B$
be the multiplication operators in $L^p(\cI,X)$ induced, respectively, by $A$ and by $B(t)$. Let
$\cl K_0 : =\overline{D_0+\cl A}$.

 \item[\rm \;\;(i)]  If the operator family $\{C(t)\}_{t\in\cl I}$ is stable (\ref{eq:5.2}), then the
 pair $\{D_0,\cC\}$ is Trotter-stable (\ref{eq:5.1}).

\item[\rm\;\;(ii)]  If the family $\{B(t)\}_{t\in\cl I}$ is $A$-stable (\ref{eq:5.3}), then the pair
$\{\cK_0,\cB\}$ is Trotter-stable (\ref{eq:5.1}).
\end{prop}
%
\begin{proof}
(i) The right-shift semigroup $\{S(\tau)\}_{\tau\geq0}$ (\S \ref{sec:3.1}) is nilpotent, and hence, the
product is zero for $\tau \ge T$. We have
 \begin{align*}
  ((e^{-\frac{\tau}{n} \cC}e^{-\frac{\tau}{n} D_0})^nf)(t) = \prod^{\rightarrow n-1}_{j=0}
  e^{-\frac{\gt}{n}C(t - j\frac{\gt}{n})}\chi_\cI(t-\gt)f(t-\gt),\quad t\in\R,
 \end{align*}
$f \in L^p(\cI,X)$, $p \in [0,\infty)$, where the product is increasingly ordered in $j$ from the left to
the right. Let us introduce the left-shift semigroup $L(\tau)$, $\tau \ge 0$,
\begin{align}\label{L-shift}
(L(\tau)f)(t) :=  \chi_{\cI}(t+\tau)f(t+\tau), \quad t\in \cI, \quad f \in L^p(\cI,X),\quad p \in [1,\infty).
\end{align}
Using this semigroup we find that
 \begin{align*}
 \left( L(\tau) \left(e^{-\frac \tau n\cC}e^{-\frac \tau n D_0}\right)^n f\right)(t) =
 \left(\prod_{j=1}^{n\leftarrow} e^{-\frac\tau n C(t+ j\frac{\tau}{n})}\right)\chi_\cI(t+\gt)f(t),
\quad t \in \cI,
 \end{align*}
for $f \in L^p(\cI,X)$, $p \in [1,\infty)$. Therefore, the operator in the left-hand side is a multiplication
operator induced by
 \begin{align*}
 \left\{\prod_{j=1}^{n\leftarrow} e^{-\frac\tau n C(t+j\frac{\tau}{n})} \chi_\cI(t+\tau)\right\}_{t\in\cI}.
 \end{align*}
 Using Proposition \ref{OperatorNormsBoundedOperator} and assuming that $\{C(t)\}_{t\in\cI}$ is
 Trotter-stable \eqref{eq:5.2}, we obtain the estimate
 \begin{align*}
\|L(\tau)\left(e^{-\frac \tau n\cC}e^{-\frac \tau n D_0}\right)^n\|_{\cl B(L^p(\cI,X))}=
\esssup_{0 \le t \le T-\gt} \left\|\prod_{j=1}^{n\leftarrow} e^{-\frac{\tau}{n} C(t+j\frac{\tau}{n}})
\right\|_{\cB(X)} \le M.
 \end{align*}
Since for $\tau \in [0,T)$ one has
\begin{align}\label{L-shift-id}
 \|\left(e^{-\frac \tau n\cC}e^{-\frac \tau n D_0}\right)^n\|_{\cl B(L^p(\cI,X))} =
 \|L(\tau)\left(e^{-\frac \tau n\cC}e^{-\frac \tau n D_0}\right)^n\|_{\cl B(L^p(\cI,X))} \ ,
 \end{align}
this estimate proves the claim (i). In a similar manner one proves the claim (ii).
\end{proof}
Now, let us introduce the operator family:
\begin{align*}
 T(\tau) = e^{-\tau\cl B}e^{-\tau  \cl K_0}, \quad \gt \ge 0.
\end{align*}
Note that if the family $\{B(t)\}_{t\in\cl I}$ is $A$-stable \eqref{eq:5.3}, then
\begin{align}\label{T<M}
 \|T\left(\frac\tau n\right)^n\|_{\cl B(L^p(\cl I,X))} \leq M, ~~{\rm for}~~ n\in\N ~~{\rm and}~~ \tau\geq 0.
\end{align}
%
\begin{lem}\label{StabilityLemma}
 If the operator family $\{B(t)\}_{t\in\cl I}$ is $A$-stable, then
\begin{align*}
 \|T\left(\frac\tau n\right)^m\|_{\cl B(L^p(\cl I,X))} \leq M
\end{align*}
for any $m \in \N$, $n \in \N$ and $\gt \ge 0$. In particular, we have
\begin{align*}
 \|T(\tau)^m\|_{\cl B(L^p(\cl I,X))} \leq M
\end{align*}
for any $m \in \N$ and $\gt \ge 0$.
\end{lem}
%
\begin{proof}
 After the change of variables: $\tau = {\sigma n}/{m}$, one proves the first statement since
 it reduces to the estimate (\ref{T<M}):
 \begin{align*}
  \sup_{\tau \geq 0} \|T\left(\frac\tau n\right)^m\|_{\cl B(L^p(\cl I,X))} = \sup_{ {\sigma n}/{m} \geq 0}
  \|T\left(\frac\sigma m\right)^m\|_{\cl B(L^p(\cl I,X))}\leq M \ .
 \end{align*}
Setting $n =1$ we get the second statement.
\end{proof}

\section{Convergence in the strong topology}\label{sec:6}
Theorem \ref{OurEvolutionProblemUniqueSolution} yields the existence and uniqueness of a solution operator
$U(t,s)$, $(t,s)\in\Delta$, for the evolution equation \eqref{EvolutionProblem}.  This solution operator
may be approximated by the product-type formulae in different operator topologies under hypothesis from
Assumption \ref{ass:1.1} and stability conditions.

We start by the claim that the classical Trotter formula can be used to prove the \textit{strong} operator
convergence in $L^p(\cI,X)$ of the product approximants for the semigroup generated by $\cl K$.
%
\begin{thm}\label{TheoTPFstrTopInFrakX}
Let the assumptions (A1), (A3) and (A4) be satisfied.
Let $\mathcal{A}$ and $\mathcal{B}$ be the induced multiplication operators in $L^p(\cI,X)$.
Define $\cl K_0:=\overline{D_0+\mathcal{A}}$ and let $\cl K=\cl K_0+\cl B$.
\item[\;\;\rm(i)] If the operator family $\{C(t)\}_{t\in\cI}$ is stable, then
 \begin{align*}
 e^{-\tau \cl K} = \slim (e^{-\frac{\tau}{n} D_0}e^{-\frac{\tau}{n}(\cl A+\cl B)})^n=
 \slim (e^{-\frac{\tau}{n}(\cl A+\cl B)}e^{-\frac{\tau}{n} D_0})^n
 \end{align*}
in the strong operator topology uniformly in $\tau\ge 0$.

\item[\;\;\rm(ii)] If the operator family $\{B(t)\}_{t\in\cI}$
is $A$-stable, then
\begin{align*}
e^{-\tau \cl K} = \slim (e^{-\frac{\tau}{n}\cl K_0}e^{-\frac{\tau}{n}\cl B})^n=
\slim (e^{-\frac{\tau}{n}\cl B}e^{-\frac{\tau}{n}\cl K_0})^n
\end{align*}
in the strong operator topology uniformly in $\tau \ge 0$.

\end{thm}
%
\begin{proof}
The proof follows immediately from Proposition \ref{KbecomesGenerator},
Proposition \ref{ClassicalTrotter} and Proposition \ref{StabilityProposition}.
\end{proof}

Theorem \ref{TheoTPFstrTopInFrakX} provides information about the strong
convergence of the Trotter product formula in $L^p(\cI,X)$.  Notice that two different operator
splittings of the operator $\cK$ yield in Theorem\ref{TheoTPFstrTopInFrakX} two different product approximations
(i) and (ii).

Let $\{\{U_n(t,s)\}_{(t,s))\in\gD}\}_{n\in\N}$ be the operator family defined by
\eqref{ApproximationPropagatorIntroduction} and let $\{\cl U(\tau)\}_{\tau\geq0}$ be the semigroup
generated by $\cl K$, i.e., $e^{-\tau \cl K}=e^{-\tau(\cl B+\cl K_0)}=\cl U(\tau)$.
Then for any $f\in  L^p(\cI,X)$ one gets
\begin{align*}
 ((e^{-\frac \tau n \cl B}e^{-\frac \tau n \cl K_0})^nf)(t)= U_n(t,t-\tau)\chi_\cI(t-\tau)f(t-\tau),
 \quad  \quad t\in \cI \ .
\end{align*}
Since
\begin{align*}
 (e^{-\tau \cl K}f)(t)=(e^{-\tau(\cl B+\cl K_0)}f)(s)=(\cl U(\tau)f)(s)=U(t,t-\tau)\chi_\cI(t-\tau)f(s-\tau) \ ,
\end{align*}
we conclude that
\begin{align}\label{eq:6.main}
 (\{(e^{-\frac \tau n \cl B}e^{-\frac \tau n \cl K_0})^n&-e^{-\tau(\cl B+\cl K_0)}\}f)(t)
 =\{U_n(t,t-\tau)-U(t,t-\tau)\}\chi_\cI(t-\tau)f(t-\tau) \ .
\end{align}
Note that the product formula in a different order yields
\begin{align*}
 (\{(e^{-\frac{\tau}{n}\cl K_0}e^{-\frac{\tau}{n}\cl B})^n&-e^{-\tau \cl K}\}f)(t)=\left\{U^\prime_n(t,t-\tau)-U(t,t-\tau)\right\}\chi_\cI(t-\tau)f(t-\tau) \ ,
\end{align*}
where the approximating propagator is given by
\begin{equation}\label{Un2}
\begin{split}
 U^\prime_n(t,s) &:=\prod^{n-1\leftarrow}_{j=0} G^\prime_j(t,s\,;n), \quad n =1,2,\ldots\;,\\
G^\prime_j(t,s\,;n) &:= e^{-\frac{t-s} n A}e^{-\frac{t-s}{n} B(s+j\frac{t-s}{n})}, \quad j = 0,1,2,\ldots,n \ .
\end{split}
\end{equation}
Note that for the case of Theorem \ref{TheoTPFstrTopInFrakX} (i) the Trotter product approximations
get the form
\begin{equation}\label{Vn1}
 V_n(t,s)  := \prod^{n\leftarrow}_{j=1} e^{-\frac{t-s} n C(s+ j\frac{t-s}{n})}
\quad \mbox{and}\quad
 V^\prime_n(t,s)  := \prod^{n-1\leftarrow}_{j=0} e^{-\frac{t-s} n C(s+j\frac{t-s}{n})} \ ,
\end{equation}
for $(t,s) \in \gD$,  $n = 1,2,\ldots$ and $C(t)=A+B(t)$.

\begin{thm}\label{TheoTPFstrTopInX}
Let the assumptions (A1), (A3) and (A4) be satisfied.

\item[\rm\;\;(i)] If the family $\{C(t)\}_{t\in \cI}$ is stable, then
\begin{equation*}
\begin{split}
0 &= \lim_{n\to\infty}\sup_{\tau\in\cI}\int_0^{T-\tau}ds \, \|\{V_n(s+\tau,s)-U(s+\tau,s)\}x\|^p_X \\
&= \lim_{n\to\infty}\sup_{\tau\in\cI}\int_0^{T-\tau}ds\, \|\{V^\prime_n((s+\tau,s)-U(s+\tau,s)\}x\|^p_X \ ,
\end{split}
\end{equation*}
for any $p\in[1,\infty)$ and $x \in X$, where the families $\{\{V_n(t,s)\}_{(t,s)\in\gD}\}_{n\in\N}$ and
$\{\{V^\prime_n(t,s)\}_{(t,s)\in\gD}\}_{n\in\N}$
are defined by \eqref{Vn1}.

\item[\rm\;\;(ii)] If the family $\{B(t)\}_{t\in\cI}$ is $A$-stable, then
\begin{equation}
\begin{split}
0 &= \lim_{n\to\infty}\, \sup_{\tau\in\cI}\int_0^{T-\tau}ds\|\{U_n(s+\tau,s)-U(s+\tau,s)\}x\|^p_X \\
&=\lim_{n\to\infty}\, \sup_{\tau\in\cI}\int_0^{T-\tau}ds\|\{U^\prime_n(s+\tau,s)-U(s+\tau,s)\}x\|^p_X \ ,
\end{split}
\end{equation}
for any $p\in[1,\infty)$ and $x \in X$, where the families $\{\{U_n(t,s)\}_{(t,s)\in\gD}\}_{n\in\N}$ and
$\{\{U'_n(t,s)\}_{(t,s)\in\gD}\}_{n\in\N}$ are defined, respectively, by
\eqref{ApproximationPropagatorIntroduction}  and \eqref{Un2}.
\end{thm}
%
\begin{proof}
We prove only the statement for $\{\{U_n(t,s)\}_{(t,s)\in \gD}\}_{n\in\N}$. The other statements
can be proved similarly.

Take $f=\phi \otimes x\in L^p(\cI)\otimes X \cong L^p(\cI, X)$ for $x\in X$ and $\phi\in L^p(\cI)$.
Then, we have
\begin{align*}
 \|&\left((e^{-\frac{\tau}{n}\cl B}e^{-\frac{\tau}{n}\cl K_0})^n-e^{-\tau K}\right)f\|^p_{L^p} =  \int_0^T ds\, \|\{U_n(s,s-\tau)-U(s,s-\tau)\}\chi_\cI(s-\tau)f(s-\tau)\|^p_X = \\
 & = \int_0^{T-\tau}ds\, \|\{U_n(s+\tau,s)-U(s+\tau,s)\}f(s)\|^p_X = \int_0^{T-\tau}ds\, |\phi(s)|^p\|\{U_n(s+\tau,s)-U(s+\tau,s)\}x\|^p_X ,
\end{align*}
$\gt \in \cI$, which proves the claim if $\phi(t)=1$ a.e. in $\cI$.
\end{proof}

\begin{rem}
 We note that the corresponding convergences in Theorem \ref{TheoTPFstrTopInFrakX} and in Theorem
 \ref{TheoTPFstrTopInX} are equivalent.
\end{rem}

\section{Convergence in the operator-norm topology} \label{sec:7}

In \S \ref{sec:7} we proved the convergence of the product approximants $U_n(t,s)$ to solution operator
$U(t,s)$ in the strong operator topology. In applicatons, a convergence in the operator-norm topology
is more useful, especially if the rate of convergence can be estimated. Then in contrast to analysis
in \cite{Batkai2011} we obtain a vector-independent estimate of accuracy of the solution approximation by
the product formulae.

In this section, we want to estimate the convergence-rate of for
$\sup_{(t,s)\in\gD}\|U(t,s)-U_n(t,s)\|_{\cl B(X)} \longrightarrow 0$, when $n\to \infty$.
An important ingredient for that is to estimate the error bound for the Trotter product formula
approximation of the evolution semigroup:
$\sup_{\gt \ge 0}\|(e^{-\tau\cK_0/n}e^{-\tau\cB/n})^n-e^{-\tau\cl K}\|_{\cl B( L^p(\cI,X))}
\longrightarrow 0$, when $n\to \infty$.

\subsection{Technical Lemmata}\label{sec:7.1}

Here we state and we prove all technical lemmata that we need for demonstration of convergence and
for estimate of the error bound for the Trotter product formula approximations in the operator-norm in
$ L^p(\cI,X)$.
%
\begin{lem}
Assume (A1), (A2) and (A5).
 Then, the operator $\overline{\cl A^{-1}\cl B}$ is  bounded on $L^p(\cI,X)$ and the norm
 $\|\overline{\cl A^{-1}\cl B}\|_{\cl B( L^p(\cI,X))}\leq C_1^*$, $p \in [1,\infty)$.
\end{lem}
%
\begin{proof}
 Let $x\in \dom(A)\subset \dom(B(t))$ and $\xi \in X^*$. Then, it holds that
 \begin{align*}
  |\la A^{-1}B(t)x, \xi \ra| = |\la x, B(t)^*(A^{-1})^*\xi \ra| \leq C_\ga^* \|x\|\,\|\xi\|,
  {\rm ~~a.e. ~t\in\cI}.
 \end{align*}
 Since $\dom(A)\subset X$ is dense, we conclude that ${\rm ~ess~sup}_{t\in \cI}
 \|\overline{A^{-1}B(t)}\|_{\cl B(X)}\leq C_1^*$.
Let $\Gamma \in L^p(\cI;X)^*$.  By Proposition \ref{CharaterizationLpdual} (iii) we find
\begin{align*}
\Gamma(\cl A^{-1}\cl Bf) = \int_{\cI}\la A^{-1}B(t)f(t),\Psi(t)\ra dt, \quad f \in \dom(\cl B).
\end{align*}
Then the estimate ${\rm ~ess~sup}_{t\in \cI}\|\overline{A^{-1}B(t)}\|_{\cl B(X)}\leq C_1^*$
implies
\begin{align*}
|\Gamma(\cl A^{-1}\cl Bf) | \le C^*_1 \|f\|_{L^p(\cI,X)}\left(\int_{\cI}\|\Psi(t)\|^{p'} dt\right)^{1/p'},
\end{align*}
$f \in \dom(\cl B)$, $\Gamma \in L^p(\cI,X)^*$, which yields
\begin{align*}
|\Gamma(\cl A^{-1}\cl Bf) | \le C^*_1 \|f\|_{L^p(\cI,X)} \|\Gamma\|_{L^p(\cI,X)^*}, \quad f \in \dom(\cl B).
\end{align*}
Hence we get $\|\cl A^{-1}\cl Bf\|_{L^p(\cI,X)} \le C^*_1\|f\|_{L^p(\cI,X)}$, $f \in \dom(\cl B)$,
which proves the claim.
\end{proof}

\begin{rem}
By assumption  $(A1)$ the operator $\cl A$ generates a holomorphic semigroup. Note that the operator
$\cl K_0$ is not a
generator of a holomorphic semigroup. Indeed, we have
 \begin{align*}
  (e^{-\tau \cl K_0}f)(t)=(e^{-\tau D_0}e^{-\tau A}f)(t)= e^{-\tau A}f(t-\tau)\chi_\cI(t-\tau),
  \quad f \in L^p(\cI,X) \ .
 \end{align*}
 Since the right-hand side is zero for $\tau\geq t$, the semigroup has no analytic extended to the complex
 plane $\mathbb{C}$.
\end{rem}

We comment that in general, $\dom(\cl K_0)\subset\dom(\cl A)$ does not hold, but we can prove the
following inclusion.
%
\begin{lem}\label{L0AalphaDomainsInclusions}
Let the assumption (A1) be satisfied. Then for $\alpha\in[0,1)$ one gets that
 \begin{equation}\label{eq:6.1}
\dom(\cl K_0)\subset\dom(\cl A^\alpha) \ .
 \end{equation}
\end{lem}
%
\begin{proof}
We know that the semigroup, which is generated by $\cl K_0$, is nilpotent, i.e. $e^{-\tau \cl K_0}=0$ holds for
$\tau\geq T$. Hence, generator $\cl K_0$ has empty spectrum. Since the semigroup is nilpotent, one gets
\begin{align*}
 \cl A^\alpha \cl K_0^{-1}f=\cl A^\alpha \int_0^\infty e^{-\tau \cl K_0}f d\tau= \int_0^T
 e^{-\tau D_0}\cl A^\alpha e^{-\tau \cl A}f d\tau, \quad f \in L^p(\cI,X) \ .
\end{align*}
For $\alpha\in[0,1)$ and $\tau>0$, we obtain $\|\cA^\alpha e^{-\tau \cl A}\|\leq {M^A_\alpha }/{\tau^\alpha}$ and
in addition:
\begin{align*}
\int_0^T \tau^{-\alpha}d\tau = \frac{T^{1-\alpha}}{1-\alpha}<\infty \ .
\end{align*}
So, the integrand $\cl A^\alpha e^{-\tau D_0}e^{-\tau \cl A}f$ for $\tau>0$ is integrable on
$[0,\infty)$, that implies the claim.
\end{proof}
%
\begin{lem}\label{TechnicalLemma1}
Let the assumptions (A1), (A2), and (A3) be satisfied.
Then there is a constant $\gL_\ga > 0$ such that
 \begin{equation}\label{eq:6.2}
  \|\cl A^\alpha e^{-\tau \cl K}\|_{\cl B( L^p(\cI,X))}\leq \frac{\gL_\ga}{\tau^\alpha}, \quad \gt > 0.
 \end{equation}
\end{lem}
%
\begin{proof}
Note that the following holds
 \begin{align*}
  \frac{d}{d\sigma} e^{-(\tau-\sigma) \cl K_0}e^{-\sigma \cl K}f &= e^{-(\tau-\sigma)\cl K_0}\cl
  K_0e^{-\sigma \cl K}f + e^{-(\tau-\sigma)\cl K_0}(-\cl K)e^{-\sigma \cl K}f=-e^{-(\tau-\sigma)\cl K_0}\cl Be^{-\sigma \cl K}f \ .
 \end{align*}
So, we get
\begin{align*}
 \int_0^\tau e^{-(\tau-\sigma)\cl K_0}\cl Be^{-\sigma \cl K}f d\sigma=e^{-\tau \cl K_0}f-e^{-\tau \cl K}f \ ,
\end{align*}
and hence
\begin{align*}
 \cl A^\alpha e^{-\tau \cl K} f = \cl A^\alpha e^{-\tau \cl K_0}f - \cl A^\alpha \int_0^\tau
 e^{-(\tau-\sigma)\cl K_0}\cl Be^{-\sigma \cl K}f d\sigma \ .
\end{align*}
Now, we estimate the two terms in the right-hand side. First we find that
\begin{align*}
 \|\cl A^\alpha e^{-\tau \cl K_0}f\|_{L^p(\cI,X)}\leq
 \|A^\alpha e^{-\tau A}f(\cdot-\tau)\|_{L^p(\cI,X)}\leq
 \frac{M^A_\alpha}{\tau^\alpha}\|f\|_{L^p(\cI,X)} \ .
\end{align*}
Now, let $f\in\dom(K)$. Since $\dom(\cl K)=\dom(\cl K_0)\subset\dom(\cl A^\alpha)$ (see Lemma
\ref{L0AalphaDomainsInclusions}), one gets
\begin{align*}
\cl A^\alpha\int_0^\tau d\sigma \, e^{-(\tau-\sigma)\cl K_0}\cl Be^{-\sigma \cl K} f
=\int_0^\tau d\sigma\,  A^\alpha e^{-(\tau-\sigma)\cl K_0}\cl B \cl A ^{-\alpha} \cl A ^\alpha
e^{-\sigma \cl K} f \ .
\end{align*}
There is a constant $C_\alpha>0$, such that $\|\cl B \cl A^{-\alpha}\|_{\cl B( L^p(\cI,X))}\leq C_\alpha$
(cf. Lemma \ref{ClBClAalphaBounded}). Then we find the estimate for the sum of two terms:
\begin{align*}
 &\|\cl A^\alpha e^{-\tau \cl K}f\|_{L^p(\cI,X)} \leq \\ 
 &\leq \frac{M^A_\alpha}{\tau^\alpha}\|f\|_{L^p(\cI,X)} +
 +\int_0^\tau  \|A^\alpha e^{-(\tau-\sigma)\cl K_0}\|_{\cl B( L^p(\cI,X))}\cdot \|\cl B \cl A ^{-\alpha}
 \|_{\cl B( L^p(\cI,X))}\cdot\| \cl A ^\alpha e^{-\sigma K} f \|_{L^p(\cI,X)} d\sigma \leq \\
 &\leq \frac{M^A_\alpha}{\tau^\alpha}\|f\|_{L^p(\cI,X)} +  C_\alpha\int_0^\tau \frac{M^A_\alpha}
 {(\tau-\sigma)^\alpha}\|\cl A^\alpha e^{-\sigma \cl K}f\|_ {L^p(\cI,X)} d\sigma \ .
\end{align*}
Let $\|f\|_{L^p(\cI,X)} \le 1$ and we introduce $F(\gt) := \|\cA^\ga e^{-\gt \cK}f\|_{L^p(\cI,X)}$. Then the
Gronwall-type inequality
\begin{align*}
 0\leq F(\tau) \leq c_1 \tau^{-\alpha}+c_2\int_0^\tau F(\sigma) (\tau-\sigma)^{-\alpha} d\sigma,
\end{align*}
is satisfied, where $c_1=M^A_\ga$ and $c_2= C_\alpha M^A_\ga$.
The Gronwall-type lemma (see Appendix, Lemma \ref{GronwallLemma}) states that the estimate
$F(\tau)\tau^\alpha\leq 2c_1$ is valid for $\tau\in [0,\tau_0]$, where
$\tau_0=\sigma_\alpha\cdot\min\left\{{1}/{c_2} ,\left({1}/{c_2}\right)^{1/(1-\alpha)}\right\}$
and
$\sigma_\alpha$ depends only on $\alpha$. Hence, we obtain
\begin{equation}\label{eq:6.3a}
 \|\cl A^\alpha e^{-\tau \cl K}f\| \leq \frac{2M^A_\alpha}{\tau^\alpha} \ ,
\end{equation}
for $\gt \in (0,\gt_0]$. Since $e^{-\gt \cK}f = 0$ for $\gt \ge T$ it remains to consider the case
$0 < \gt_0 < T$. If $\gt \in (\gt_0,T]$, then we find
\begin{displaymath}
 \|\cl A^\alpha e^{-\tau \cl K}f\| \le  \|\cl A^\alpha e^{-\tau_0 \cl K}e^{-(\gt-\gt_0)\cK}f\| \le
 \frac{2M^A_\alpha \,M_\cK}{\tau^\alpha_0} \ ,
\end{displaymath}
where $\|e^{-\gt \cK}f\| \le M_\cK$ for $\gt \ge 0$ and $\|f\|_{L^p(\cI,x)} \le 1$.  Here we have used
that any evolution semigroup is bounded. Hence
\begin{displaymath}
 \|\cl A^\alpha e^{-\tau \cl K}f\| \le  \frac{\gt^\ga}{\gt^\ga_0}\frac{2M^A_\alpha \,M_\cK}{\tau^\alpha}
 \le \frac{T^\ga}{\gt^\ga_0}\frac{2M^A_\alpha \,M_\cK}{\tau^\alpha} \ ,
\end{displaymath}
for $\gt \in (\gt_0,T]$ and $\|f\|_{L^p(\cI,X)} \le 1$. Setting $\gL_\ga :=
\max\left\{2M^A_\ga, ({2M^A_\alpha \,M_\cK T^\ga})/{\gt^\ga_0}\right\}$ and
taking the supremum over the unit ball we complete the proof.
\end{proof}
\begin{lem}\label{TechnicalLemma2}
Let the assumptions (A1), (A3) and (A4) be satisfied . If the family of generators $\{B(t)\}_{t\in\cI}$
is $A$-stable, then there exist constants $c_1 > 0$ and $c_2>0$ such that
\begin{equation}\label{eq:6.3}
 \|\overline{T(\gt)^k\cl A}\|_{\cl B( L^p(\cI,X))}\leq \frac{c_1}{\tau^\ga} +\frac{c_2}{k\tau},
 \quad \gt > 0 \ ,
 \quad k \in \N.
\end{equation}
\end{lem}
%
\begin{proof}
For $k\gt \ge T$ we have $T(\gt)^k = 0$. Hence, one has to prove the estimate \eqref{eq:6.3} only for
$k\gt \le T$. By Lemma \ref{StabilityLemma} we get that $\|T(\gt)^k\|_{\cl B( L^p(\cI,X))}\leq M$
for some positive constant $M$. Let $f\in\dom(\cA)$. Then
 \begin{align*}
\|T(\gt)^k\cA f\| &\leq \|(T(\gt)^k-e^{-k\tau \cl K_0})\cl Af\| + \|e^{-k\tau \cl K_0}\cA f\|\\
& \leq \|\sum_{j=0}^{k-1}T(\gt)^{k-(j+1)}(e^{-\tau \mathcal{B}}-I)e^{-(j+1)\tau \cl K_0}\cA f\|  +
\|e^{-k\tau \cl K_0}\cA f\|\\
& \leq M \sum_{j=0}^{k-1}\int_0^\tau d\sigma\|e^{-\sigma\mathcal{B}}\mathcal{B}{\cl A}^{-\alpha}\|\;
\|\cA^{1+\alpha}e^{-(j+1)\tau \cl K_0}f\| +\|e^{-k\tau \cl K_0}\cA f\|,
\end{align*}
where we used $I-e^{-\tau \cl B}=\int_0^\tau d\sigma \, \cl Be^{-\sigma\cl B}$. We have  $\|e^{-\sigma\mathcal{B}}\|\leq M_B e^{\beta_\cl B \sigma}\leq  M_B e^{\beta_c\ B T} =: M_B'$. By Proposition
\ref{HolomorphicSemigroupEstimate} we get
\begin{align*}
 \|\cA^{1+\alpha}e^{-(j+1)\tau \cl K_0}f\|&\leq\frac{M^A_{1+\ga}}{((j+1)\tau)^{1+\alpha}}\|f\| \ ,\\
 \|\cl A e^{-k\tau \cl K_0}f\|&\leq\frac{M^A_\ga}{k\tau}\|f\| \ .
\end{align*}
Therefore, one obtains the estimate:
\begin{align*}
 \|T(\gt)^k\cl Af\|
 &\leq \frac{ M M_B'M^A_{1+\alpha}  C_\alpha\tau }{\tau^{\alpha+1}}\sum_{j=0}^{k-1}
 \frac{1}{(j+1)^{\alpha+1}}\|f\|+\frac{ M^A_1}{k\tau}\|f\|\\
 &\leq \frac{MM_B' M^A_{1+\alpha} C_\alpha \zeta(\alpha+1)}{\tau^\alpha}\|f\|+\frac{M^A_1}{k\tau}\|f\| \ ,
\end{align*}
where $\zeta(\alpha+1) :=\sum_{j=1}^{\infty}{1}/{j^{\alpha+1}}$ is the Riemann $\zeta$-function.
Since $T(\gt)^k = 0$ for $\gt k \ge T$, we find
\begin{align*}
 \|T(\gt)^k\cl Af\| \leq \frac{MM_B' M^A_{1+\alpha}  C_\alpha \zeta(\alpha+1)}{\tau^\alpha}\|f\|+
 \frac{M^A_1}{k\tau}\|f\|, \quad f \in \dom(\cA) \ .
\end{align*}
Then estimate \eqref{eq:6.3} follows by taking supremum over the unit ball in $\dom(\cA)$ .
\end{proof}
%
\begin{lem}\label{TechnicalLemma3}
Let the assumptions (A1), (A3), (A4), and (A5) be satisfied. Then there is a constant $c>0$ such
that for $\tau \geq 0$ we have inequalities {{\rm :}}
\begin{equation*}
\|(T(\gt)-e^{-\tau \cl  K})\cA^{-\ga}\|_{\cl B(L^p(\cI,X))}\leq c\tau ~~{\rm and~~}
\|\cl A^{-1}(T(\gt)-e^{-\tau \cl K})\|_{\cl B(L^p(\cI,X))}\leq
c\tau \ .
\end{equation*}
\end{lem}
%
\begin{proof}
(i) Using the representation
\begin{displaymath}
T(\gt)g -e^{-\tau \cl K}g =(e^{-\tau\cl B}-I)e^{-\tau \cl K_0}g + e^{-\tau \cl K_0}-e^{-\tau \cl K}g,
\quad g \in L^p(\cI,X),
\end{displaymath}
we get
\begin{equation}\label{eq:6.6}
 \|(T(\gt)-e^{-\tau \cl K})\cl A^{-\alpha}f\|\\
 \leq \; \|(e^{-\tau\cl B}-I)\cl A^{-\alpha}e^{-\tau \cl K_0}f\| + \|(e^{-\tau \cl K_0}-e^{-\tau \cl K})
 {\cl A}^{-\alpha}f\|
\end{equation}
for $\tau \ge 0$, $g = \cA^{-\ga}f$ and $f \in L^p(\cI,X)$. From the representation
\begin{align*}
(I - e^{-\tau\cl B})g = \int^\tau_0 e^{-s \cl \cl B} \cl B\, g \,ds, \qquad g \in \dom(\cl B),
\end{align*}
we obtain
\begin{align*}
\|(I - e^{-\tau\cl B}){\cl A}^{-\alpha} g\| \le \tau\,\|\cl B {\cl A}^{-\alpha}\| \,M_{\cl B} \,
e^{\beta_{\cl B} T}\|g\|\qquad g \in \dom(\cl B), \quad \tau \ge 0 \ .
\end{align*}
Then setting $g = e^{-\tau \cl K_0}f$, $f \in L^p(\cI,X)$, and using the estimate $\|e^{-\tau \cl K_0}f\|
\le M_{\cl A}e^{\beta_{\cl A}T}$, $\tau \ge 0$, one gets for the first term in the right-hand side of
(\ref{eq:6.6}) the estimate
\begin{align}\label{eq:6.7}
\|(I - e^{-\tau\cl B}){\cl A}^{-\alpha} e^{-\tau \cl K_0}f\| \le \tau\, C_\alpha \,M_{\cl B} \,
M_{\cl A}\, e^{\gb_\cB T}\|f\|, \quad \tau \ge 0 \ ,
\end{align}
where we also used that $e^{-\gt \cK_0} = 0$ for $\gt \ge T$. To estimate the second term note that
\begin{align*}
 e^{-\tau \cl K}g -e^{-\tau \cl K_0}g =
\int_0^\tau\frac{d}{d\sigma}\{e^{-\sigma \cl K}e^{-(\tau-\sigma) \cl K_0}\}g d\sigma
= -\int_0^\tau e^{-\sigma \cl K}\cl B e^{-(\tau-\sigma ) \cl K_0}g d\sigma \ ,
\end{align*}
$g \in \dom (\cl K_0)$. Let $g \in  {\cl A}^{-\ga}f$, $f \in \cl A^\ga \,\dom(\cl K_0)$. Then
\begin{align*}
 (e^{-\tau \cl K} -e^{-\tau \cl K_0}){\cl A}^{-\ga}f  =
-\int_0^\tau e^{-\sigma \cl K}\cl B e^{-(\tau-\sigma ) \cl K_0}{\cl A}^{-\ga}f d\sigma \ ,
\end{align*}
which leads to the estimate
\begin{align}\label{eq:6.8}
 \|(e^{-\tau \cl K} -e^{-\tau \cl K_0}){\cl A}^{-\ga}f\|  \le
\tau \, C_\alpha M_{\cl K}\,M_{\cl A} \|f\|, \qquad \tau \ge 0 \ ,
\end{align}
$f \in \cl A^\ga\,\dom(\cl K_0)$. Since $\cl A^\ga\,\dom(\cl K_0)$ is dense in $L^p(\cI,X)$ the estimate
extends to $f \in L^p(\cI,X)$.
Taking into account \eqref{eq:6.6}, \eqref{eq:6.7}, and\eqref{eq:6.8} we get the first of the claimed
in lemma inequalities.

(ii) To prove the second inequality we note that
\begin{align}\label{eq:6.9}
 \|&\cl A^{-1}(e^{-\tau \cl B}e^{-\tau \cl K_0}-e^{-\tau \cl K})f\| \leq \|\cl A^{-1}(e^{-\tau\cl B}-I) e^{-\tau \cl K_0}f\|+\|\cl A^{-1}(e^{-\tau \cl K_0}-e^{-\tau \cl K})f\|
\end{align}
$f \in \dom(\cl K_0) = \dom(\cl K)$. Using
\begin{align*}
\cl A^{-1}(I - e^{-\tau\cl B}) e^{-\tau \cl K_0}f = \int^\tau_0 d\sigma \, \cl A^{-1}\cl B
e^{-\sigma \cl B}e^{-\tau \cl K_0}f \ , 
\end{align*}
one finds the estimate for the first term in the right-hand side of (\ref{eq:6.9}):
\begin{align}\label{eq:6.10}
\|\cl A^{-1}(I - e^{-\tau\cl B}) e^{-\tau \cl K_0}f\| \le \tau\, C^*_1 M_{\cl B}\,M_{\cl A}\,
e^{\gb_\cB T}\,\|f\|, \quad \tau \ge 0 \ .
\end{align}
For the second term we start with identity
\begin{align*}
 e^{-\tau \cl K_0}f -e^{-\tau \cl K}f =
\int_0^\tau\frac{d}{d\sigma}\{e^{-\sigma \cl K}e^{-(\tau-\sigma) \cl K}\}fd\sigma
=\int_0^\tau d\sigma \, e^{-\sigma \cl K_0}\cl B e^{-(\tau-\sigma ) \cl K}f \ ,
\end{align*}
which leads to
\begin{align*}
 \cl A^{-1}(e^{-\tau \cl K_0} -e^{-\tau \cl K})f
=\int_0^\tau d\sigma \, \cl A^{-1}\cl Be^{-\sigma \cl K_0}e^{-(\tau-\sigma ) \cl K}f  \ ,
\end{align*}
for any $f \in \dom(\cl K) = \dom(\cl K_0)$. Hence, we get the estimate
\begin{align}\label{eq:6.11}
 \|\cl A^{-1}(e^{-\tau \cl K_0} -e^{-\tau \cl K})f\|
\le \tau \,C^*_1 M_{\cl A}\,M_{\cl K}\|f\|, \quad \tau \ge 0 \ .
\end{align}
Summarising now \eqref{eq:6.9}, \eqref{eq:6.10}, and \eqref{eq:6.11}, 
we obtain the second of the claimed in lemma inequalities.
\end{proof}
%
\begin{lem}\label{TechnicalLemma4}
Let the assumptions (A1), (A3), (A4), (A5), and (A6) be satisfied.
If $\gb \in (\ga,1)$, then there exists a constant $Z(\gb) > 0$ such that
\begin{equation}
\|\cl A^{-1}(T(\tau)-e^{-\tau \cl K})\cl A^{-\gb}\|_{\cl B( L^p(\cI,X))} \leq Z(\gb) \tau^{1+\gb},
\quad \gt \ge 0.
\end{equation}
\end{lem}
%
\begin{proof}
Let $f\in \dom(\cl K_0)=\dom(\cl K)$. Then identity
\begin{align*}
 \frac{d}{d\sigma }&T(\sigma )e^{-(\tau-\sigma )\cl K}f = \frac{d}{d\sigma }e^{-\sigma  \cl B}
 e^{-\sigma  \cl K_0}e^{-(\tau-\sigma )\cl K}f \\
 =&-e^{-\sigma \cl B}\cl Be^{-\sigma \cl K_0}e^{-(\tau-\sigma )\cl K}f - e^{-\sigma \cl B}
 e^{-\sigma \cl K_0}\cl K_0e^{-(\tau-\sigma )\cl K}f +e^{-\sigma \cl B}e^{-\sigma \cl K_0}Ke^{-(\tau-\sigma )\cl K}f\\
 =&-e^{-\sigma \cl B}\cl Be^{-\sigma \cl K_0}e^{-(\tau-\sigma )\cl K}f + e^{-\sigma \cl B}
 e^{-\sigma \cl K_0}\cl B e^{-(\tau-\sigma )\cl K}f\\
 =& e^{-\sigma \cl B}\{e^{-\sigma \cl K_0}\cl Bf - \cl Be^{-\sigma \cl K_0}\} e^{-(\tau-\sigma )\cl K}f,
\end{align*}
yields
\begin{align}\label{eq:6.13a}
 &T(\tau)f-e^{-\tau \cl K}f = \int_0^\tau\frac{d}{d\sigma }T(\sigma )e^{-(\tau-\sigma )\cl K}fd\sigma=\int_0^\tau e^{-\sigma \cl B}\{e^{-\sigma \cl K_0}\cl B - \cl Be^{-\sigma \cl K_0}\}
 e^{-(\tau-\sigma )\cl K}fd\sigma .
\end{align}

On the other hand we also have the following identity:
\begin{align*}
 e^{-\sigma \cl B}&\left(e^{-\sigma \cl K_0}\cl B - \cl Be^{-\sigma \cl K_0}\right) e^{-(\tau-\sigma )
 \cl K} f\\
 =&\;(e^{-\sigma \cl B}-I)\{e^{-\sigma \cl K_0}\cl B-\cl Be^{-\sigma \cl K_0}\}(e^{-(\tau-\sigma )
 \cl K}-e^{-(\tau-\sigma )\cl K_0})f+\\
 &+(e^{-\sigma \cl B}-I)\{e^{-\sigma \cl K_0}\cl B-\cl Be^{-\sigma \cl K_0}\}e^{-(\tau-\sigma )
 \cl K_0}f+\\
 &+\{e^{-\sigma \cl K_0}\cl B-\cl Be^{-\sigma \cl K_0}\}(e^{-(\tau-\sigma )\cl K}-e^{-(\tau-\sigma )
 \cl K_0})f+\{e^{-\sigma \cl K_0}\cl B-\cl Be^{-\sigma \cl K_0}\}e^{-(\tau-\sigma )\cl K_0}f \ ,
\end{align*}
which yields for $f=\cl A^{-\gb}g$
\begin{equation}\label{eq:6.13}
\begin{split}
 \cl A^{-1}&e^{-\sigma \cl B}\left(e^{-\sigma \cl K_0}\cl B - \cl Be^{-\sigma \cl K_0}\right)
 e^{-(\tau-\sigma )\cl K}\cl A^{-\gb}g=\\
 =&\;\cl A^{-1}(e^{-\sigma \cl B}-I)\{e^{-\sigma \cl K_0}\cl B-\cl Be^{-\sigma \cl K_0}\}
 (e^{-(\tau-\sigma )\cl K}-e^{-(\tau-\sigma )\cl K_0})\cl A^{-\gb}g+\\
 &+\cl A^{-1}(e^{-\sigma \cl B}-I)\{e^{-\sigma \cl K_0}\cl B-\cl Be^{-\sigma \cl K_0}\}\cl A^{-\gb}
 e^{-(\tau-\sigma )\cl K_0}g+\\
 &+\cl A^{-1}\{e^{-\sigma \cl K_0}\cl B-\cl Be^{-\sigma \cl K_0}\}(e^{-(\tau-\sigma )\cl K}-e^{-(\tau-\sigma )
 \cl K_0})\cl A^{-\gb}g+\\
 &+\cl A^{-1}\{(e^{-\sigma \cl K_0}-e^{-\sigma D_0})\cl B-\cl B(e^{-\sigma \cl K_0}-e^{-\sigma D_0})\}
 e^{-(\tau-\sigma )\cl K_0}\cl A^{-\gb}g+\\
 &+\cl A^{-1}(e^{-\sigma D_0}\cl B-\cl Be^{-\sigma  D_0})\cl A^{-\gb}e^{-(\tau-\sigma )\cl K_0}g.
\end{split}
\end{equation}
In the following, we estimate separately the five terms in the right-hand side of 
identity (\ref{eq:6.13}).

To this end we note that $\cl A$ and $\cl K_0$ commute. This implies that
\begin{align*}
 (e^{-(\tau-\sigma )\cl K}-e^{-(\tau-\sigma )\cl K_0})\cl A^{-\gb}g=\int_0^{\tau-\sigma }dr \,
 e^{-(\tau-\sigma -r)\cl K}\cl B\cl A^{-\gb}e^{-r\cl K_0}g \ .
\end{align*}
Thus, for the first term we get
\begin{align*}
 \cl A^{-1}&(e^{-\sigma \cl B}-I)\{e^{-\sigma \cl K_0}\cl B-\cl Be^{-\sigma \cl K_0}\}
 (e^{-(\tau-\sigma )\cl K}-e^{-(\tau-\sigma )\cl K_0})\cl A^{-\gb}g\\
 &=-\int_0^\sigma dr \, \cl A^{-1}\cl Be^{-r\cl B}\,[e^{-\sigma \cl K_0},\cl B]\cl A^{-\gb}\time\sigma
 \int_0^{\tau-\sigma }dr \, \cl A^{\gb}e^{-(\tau-\sigma -r)\cl K}\cl B\cl A^{-\gb}e^{-r\cl K_0}g \ ,
\end{align*}
where
\begin{align*}
[e^{-\sigma \cl K_0},\cl B]f := \{e^{-\sigma \cl K_0}\cl B-\cl Be^{-\sigma \cl K_0}\}f, \quad
f \in \dom(\cl  K_0), \quad \tau \ge 0.
\end{align*}
Then using Lemma \ref{TechnicalLemma1}, we obtain the estimate
\begin{equation}\label{eq:6.14}
\begin{split}
 \|\cl A^{-1}(e^{-\sigma \cl B}&-I)\{e^{-\sigma \cl K_0}\cl B-\cl Be^{-\sigma \cl K_0}\}
 (e^{-(\tau-\sigma )K}-e^{-(\tau-\sigma )\cl K_0})\cl A^{-\gb}g\| \\
 &\leq \sigma\; 2C_1^*C_\gb^2\gL_\gb M_{\cl B}M^2_{\cl A}e^{\gb_\cB T}\int_0^{\tau-\sigma }dr \,
 \frac{1}{(\tau-\sigma -r)^\gb} \; \|g\|\\
 &\leq \sigma(\tau-\sigma )^{1-\gb}\; \frac{2C_1^*C_\gb^2\gL_\gb M_{\cl B}M^2_{\cl A}e^{\gb_\cB T}}
 {1-\gb}\;\|g\|
\end{split}
\end{equation}
for $\gs \in [0,\gt]$ and $\gt \ge 0$. 

For the second term, one can readily establish the estimate
\begin{equation}\label{eq:6.15}
\begin{split}
\|\cl A^{-1}&(e^{-\sigma \cl B}-I)\{e^{-\sigma \cl K_0}\cl B-\cl Be^{-\sigma \cl K_0}\}\cl A^{-\gb}
e^{-(\tau-\sigma )\cl K_0}g\|\leq \sigma \;2C_1^*C_\gb M_\cB M^2_\cA e^{\gb_\cB T}\|g\| \ ,
\end{split}
\end{equation}
for $\gs \in [0,\gt]$ and $\gt \ge 0$. 

Now note that by virtue of relation
\begin{align*}
 e^{-(\tau-\sigma )\cl K}-e^{-(\tau-\sigma )\cl K_0}h=\int_0^{\tau-\sigma }dr \,
 e^{-(\tau-r-\sigma )K}\cl Be^{-r\cl K_0}h \ , \quad h\in \dom(\cK_0) \ ,
\end{align*}
one obtains for the third term the estimate
\begin{equation}\label{eq:6.16}
\begin{split}
 \|\cl A^{-1}&\{e^{-\sigma \cl K_0}\cl B-\cl Be^{-\sigma \cl K_0}\}(e^{-(\tau-\sigma )K}-e^{-(\tau-\sigma )
 \cl K_0})\cl A^{-\gb}g\|\leq (\tau-\sigma )\; 2C_1^*C_\gb M^2_\cA M_\cK \;\|g\| \ ,
\end{split}
\end{equation}
for $\gs \in [0,\gt]$ and $\gt \ge 0$. 

Moreover, using the equality
\begin{align*}
 e^{-\sigma \cl K_0}-e^{-\sigma D_0}h =-\int_0^\sigma dr \, e^{-r\cl K_0}\cl A e^{-(\sigma -r)D_0}h \ ,
\end{align*}
we get for the fourth term:
\begin{align*}
 &\cl A^{-1}\{(e^{-\sigma \cl K_0}-e^{-\sigma D_0})\cl B-\cl B(e^{-\sigma \cl K_0}-e^{-\sigma D_0})\}
 e^{-(\tau-\sigma )\cl K_0}\cl A^{-\gb}g =\\
 &\Bigl(-\int_0^\sigma dr \, e^{-r\cl K_0}e^{-(\sigma -r)D_0} \cl B\cl A^{-\gb}+
 \cl A^{-1}\cl B\int_0^\sigma dr \, e^{-r\cl K_0}\cl A^{1-\gb}e^{-(\sigma -r)D_0}\Bigr)
 e^{-(\tau-\sigma )\cl K_0}g \,
\end{align*}
which yields the estimate
\begin{equation}\label{eq:6.17}
\begin{split}
 &\|\cl A^{-1}\{(e^{-\sigma \cl K_0}-e^{-\sigma D_0})\cl B-\cl B(e^{-\sigma \cl K_0}-e^{-\sigma D_0})\}
 e^{-(\tau-\sigma )\cl K_0}\cl A^{-\gb}g\|\\
 &\leq \sigma\; C_\alpha M_\cA  \|g\| +C_1^*M_\cA \,M^A_{1-\gb}\int_0^\sigma dr \, \frac{1}{r^{1-\gb}}
 \;\|g\| = \left(\gs \;C_\gb M_\cA+ \gs^\gb \;\frac{C_1^*\,M_\cA\,M^A_{1-\gb}}{\gb}\right)\; \|g\|
\end{split}
\end{equation}
for $\gs \in [0,\gt]$ and $\gt \ge 0$. 

To estimate the fifth term, we note that
\begin{align*}
 (e^{-\sigma D_0}\cl B&-\cl Be^{-\sigma D_0})f=e^{-\sigma D_0}B(\cdot)f(\cdot)-\cl B \chi_\cI(\cdot-\sigma )
 f(\cdot-\sigma )=\\
 &=\chi_\cI(\cdot-\sigma )B(\cdot-\sigma )f(\cdot-\sigma )-B(\cdot)\chi_\cI(\cdot-\sigma )f(\cdot-\sigma )=\\
 &=\chi_\cI(\cdot-\sigma )\{B(\cdot-\sigma )-B(\cdot)\}f(\cdot-\sigma ) \ ,
\end{align*}
and therefore, one gets
\begin{equation}\label{eq:6.18}
\begin{split}
 \|\cl A^{-1}(e^{-\sigma D_0}\cl B&-\cl Be^{-\sigma D_0})\cl A^{-\gb}e^{-(\tau-\sigma )\cl K_0}g\|=\|\cl A^{-1}\{e^{-\sigma D_0}\cl B-\cl Be^{-\sigma D_0}\}\cl A^{-\gb}g\|\\
 \leq &\esssup_{t\in\cI}\|A^{-1}\{B(t-\sigma )-B(t)\}A^{-\gb}\|_{\cl B(X)}\;\|g\|\leq L_{\beta}
 \sigma ^{\beta}\|g\|.
\end{split}
\end{equation}
for $\gs \in [0,\gt]$ and $\gt \ge 0$. 

From identity \eqref{eq:6.13} we deduce the estimate
\begin{displaymath}
\begin{split}
 \|\cl A^{-1}&e^{-\sigma \cl B}\left(e^{-\sigma \cl K_0}\cl B - \cl Be^{-\sigma \cl K_0}\right)
 e^{-(\tau-\sigma )\cl K}\cl A^{-\gb}g\|\\
\le &\;\|\cl A^{-1}(e^{-\sigma \cl B}-I)\{e^{-\sigma \cl K_0}\cl B-\cl Be^{-\sigma \cl K_0}\}(e^{-(\tau-\sigma )
\cl K}-e^{-(\tau-\sigma )\cl K_0})\cl A^{-\gb}g\|\\
 &+\|\cl A^{-1}(e^{-\sigma \cl B}-I)\{e^{-\sigma \cl K_0}\cl B-\cl Be^{-\sigma \cl K_0}\}\cl A^{-\gb}
 e^{-(\tau-\sigma )\cl K_0}g\|\\
 &+\|\cl A^{-1}\{e^{-\sigma \cl K_0}\cl B-\cl Be^{-\sigma \cl K_0}\}(e^{-(\tau-\sigma )\cl K}-e^{-(\tau-\sigma )
 \cl K_0})\cl A^{-\gb}g\|\\
 &+\|\cl A^{-1}\{(e^{-\sigma \cl K_0}-e^{-\sigma D_0})\cl B-\cl B(e^{-\sigma \cl K_0}-e^{-\sigma D_0})\}
 e^{-(\tau-\sigma )\cl K_0}\cl A^{-\gb}g\| \\
 &+\|\cl A^{-1}(e^{-\sigma D_0}\cl B-\cl Be^{-\sigma  D_0})\cl A^{-\gb}e^{-(\tau-\sigma )\cl K_0}g\| \ .
\end{split}
\end{displaymath}
for $\gs \in [0,\gt]$ and $\gt \ge 0$. Now taking into account \eqref{eq:6.14}, \eqref{eq:6.15}, 
\eqref{eq:6.16}, \eqref{eq:6.17}, and \eqref{eq:6.18} we find the estimate
\begin{align*}
 \|\cl A^{-1}&e^{-\sigma \cl B}\left(e^{-\sigma \cl K_0}\cl B - \cl Be^{-\sigma \cl K_0}\right)
 e^{-(\tau-\sigma )\cl K}\cl A^{-\alpha}g\|\leq\\
 \leq \Bigl\{&\sigma (\tau-\sigma )^{1-\alpha}\frac{2C_1^*C_\gb^2\gL_\gb M_{\cl B}
 M^2_{\cl A}e^{\gb_\cB T}}{1-\gb} +  \sigma\; 2C_1^*C_\gb M_\cB M^2_\cA e^{\gb_\cB T}+\\
 &(\tau-\sigma )\;2C_1^*C_\gb M^2_\cA M_\cK + \sigma \;C_\gb M_\cA + \sigma ^\gb \;
 \frac{C_1^*\,M_\cA\,M^A_{1-\gb}}{\gb}+ \sigma^\gb \;L_\beta\Big\}\|g\|
\end{align*}
for $\gs \in [0,\gt]$ and $\gt \ge 0$. Then setting
\begin{align*}
Z_1 &:= \frac{2C_1^*C_\gb^2\gL_\gb M_{\cl B}M^2_{\cl A}e^{\gb_\cB T}}{1-\gb}\\
Z_2 &:= 2C_1^*C_\gb M_\cB M^2_\cA e^{\gb_\cB T} + C_\gb M_\cA\\
Z_3 &:= 2C_1^*C_\gb M^2_\cA M_\cK \\
Z_4 &:=  \frac{C_1^*\,M_\cA\,M^A_{1-\gb}}{\gb} + L_\gb
\end{align*}
we obtain
\begin{align}\label{eq:6.20}
 \|\cl A^{-1}&e^{-\sigma \cl B}\left(e^{-\sigma \cl K_0}\cl B - \cl Be^{-\sigma \cl K_0}\right)
 e^{-(\tau-\sigma )\cl K}\cl A^{-\gb}g\|
 \leq \Bigl\{Z_1\;\sigma (\tau-\sigma )^{1-\gb} +  Z_2\;\sigma+ Z_3\; (\tau-\sigma )+ Z_4 \;
 \sigma ^\gb\Big\}\|g\|.
\end{align}

Now we remark that \eqref{eq:6.13a} gives the representation
\begin{align*}
 \cA^{-1}&(T(\tau)-e^{-\tau \cl K})\cA^{-\gb}g =\int_0^\tau d\sigma \, \cA^{-1}e^{-\sigma \cl B}\{e^{-\sigma \cl K_0}\cl B - \cl Be^{-\sigma \cl K_0}\}
 e^{-(\tau-\sigma )\cl K}\cA^{-\gb}g   \ ,
\end{align*}
which yields the estimate
\begin{align*}
 \|\cA^{-1}&(T(\tau)-e^{-\tau \cl K})\cA^{-\gb}g\|\leq\int_0^\tau d\sigma \, \|\cA^{-1}e^{-\sigma \cl B}\{e^{-\sigma \cl K_0}\cl B - 
 \cl Be^{-\sigma \cl K_0}\} e^{-(\tau-\sigma )\cl K}\cA^{-\gb}g\| \ .
\end{align*}
The inserting \eqref{eq:6.20} into this estimate and using
\begin{align*}
 \int_0^\tau \sigma \,(\tau-\sigma )^{1-\gb}d\sigma  =\tau^{3-\gb}\int_0^1 \,
 x(1-x)^{1-\gb}dx=\tau^{3-\gb} B(2, 2-\beta),
\end{align*}
(where $B$ is the \textit{Beta-function}), we find for $\gt \ge 0$ the estimate
\begin{displaymath}
 \|\cA^{-1}(T(\tau)-e^{-\tau \cl K})\cA^{-\gb}g\|  \le Z_1 B(2, 2-\beta) \;\tau^{3-\gb} +
 \frac{Z_2 + Z_3}{2}\; \gt^2 + \frac{Z_4}{1+\gb}\;\gt^{1+\gb} \ ,
\end{displaymath}
and consequently 
\begin{displaymath}
 \|\cA^{-1}(T(\tau)-e^{-\tau \cl K})\cA^{-\gb}g\| \le \left(Z_1 B(2, 2-\beta)\gt^{2-2\gb} +
 \frac{Z_2 + Z_3}{2}\gt^{1-\gb}+ Z_4\right)\;\gt^{1+\gb} \ .
\end{displaymath}
Since $T(\gt) = 0$ and $e^{-\gt \cK} = 0$ for $\gt \ge T$ we finally obtain
\begin{displaymath}
 \|\cA^{-1}(T(\tau)-e^{-\tau \cl K})\cA^{-\gb}g\| \le \left(Z_1B(2, 2-\beta)T^{2-2\gb} + \frac{Z_2 + Z_3}{2}
 T^{1-\gb} + Z_4\right)\;\gt^{1+\gb} \ ,
\end{displaymath}
which proves the lemma.
\end{proof}

\subsection{The Trotter product formula in operator-norm topology} \label{sec:7.2}

 \begin{thm}\label{TheoremEstimate}
 Let the assumptions  (A1), (A3), (A4), (A5), and (A6) be satisfied. If the family of generators
$\{B(t)\}_{t\in\cI}$ is $A$-stable and $\gb \in (\ga,1)$,  then there exists a constant $C_{\alpha, \beta}>0$
such that
 \begin{equation}\label{6.22aa}
  \| (e^{-\gt \cl B/n}e^{-\gt \cK_0/n})^{n}-e^{-\tau\cK}\|_{\cl B( L^p(\cI,X))}\leq  
  \frac{C_{\alpha,\beta}}{n^{\gb-\ga}} \ ,
\end{equation}
for $\gt \ge 0$ and $n = 2,3,\ldots \ $ .
\end{thm}
%
\begin{proof}
Let $T(\gs) := e^{-\gs \cB}e^{-\gs \cK_0}$ and $U(\gs) :=e ^{-\gs \cl K}$, $\gs \ge 0$. Then the following
identity holds
\begin{align*}
T(\sigma)^n&-U(\sigma)^n =\sum_{m=0}^{n-1} T(\sigma)^{n-m-1}(T(\sigma)-U(\sigma))U(\sigma)^m\\
 = &\;T(\sigma)^{n-1}(T(\sigma)-U(\sigma))+(T(\sigma)-U(\sigma))U(\sigma)^{n-1}+\sum_{m=1}^{n-2} T(\sigma)^{n-m-1}(T(\sigma)-U(\sigma)) U(\sigma)^m\\
 = &\;T(\sigma)^{n-1}\cl A\cl A^{-1}(T(\sigma)-U(\sigma))+(T(\sigma)-U(\sigma))\cl A^{-\alpha}
 \cl A^{\alpha}U(\sigma)^{n-1}+\\
 &+\sum_{m=1}^{n-2} T(\sigma)^{n-m-1}\cl A \cl A^{-1}(T(\sigma)-U(\sigma))\cl A^{-\gb} \cl A^{\gb} 
 U(\sigma)^m \ .
\end{align*}
It easily yields the inequality
\begin{align*}
 \|T(\sigma)^n-U(\sigma)^n\|
 \le& \;\|\overline{T(\sigma)^{n-1}\cl A}\|\; \|A^{-1}(T(\sigma)-U(\sigma))\| + \|(T(\sigma)-U(\sigma))
 \cl A^{-\ga}\| \;\|\cl A^{\ga}U(\sigma)^{n-1}\|+\\
&\;+\sum_{m=1}^{n-2} \|\overline{T(\sigma)^{n-m-1}\cl A}\|\; \|\cl A^{-1}(T(\sigma)-U(\sigma))\cl A^{-\gb}\|\;
\|\cl A^{\gb} U(\sigma)^m\| \ .
\end{align*}
From Lemma \ref{TechnicalLemma2} we get for $\gs \in (0,\gt]$ and $n \ge 2$ the estimates
\begin{equation}\label{eq:6.21}
\|\overline{T(\gs)^{n-1}\cl A}\| \le \frac{c_1}{\gs^\ga} +\frac{c_2}{(n-1)\gs} \ .
\end{equation}
Now, from Lemma \ref{TechnicalLemma3} we find that
\begin{displaymath}
\|\cA^{-1}(T(\sigma)-U(\sigma))\| \le c \, \gs \quad \mbox{and} \quad
\|(T(\sigma)-U(\sigma))\cl A^{-\alpha}\| \le c \, \gs \ .
\end{displaymath}
This implies that
\begin{displaymath}
\|\overline{T(\gs)^{n-1}\cl A}\|\;\|\cA^{-1}(T(\sigma)-U(\sigma))\| \le c_1 c\,\gs^{1-\ga} + 
\frac{c_2\,c}{n-1} \ , 
\end{displaymath}
and
\begin{displaymath}
\|(T(\sigma)-U(\sigma))\cl A^{-\alpha}\| \;\|\cl A^{\alpha}U(\sigma)^{n-1}\| \le 
\frac{c \,\gL_\ga}{(n-1)^\ga}\; \gs^{1-\ga} \ ,
\end{displaymath}
where we used also Lemma \ref{TechnicalLemma1}, \eqref{eq:6.2}. Since by Lemma \ref{TechnicalLemma4}
 \begin{align*}
  \|\cl A^{-1}(T(\gs)-e^{-\gs \cl K})\cl A^{-\gb}\|_{\cl B( L^p(\cI,X))}\leq Z(\gb)\;\gs^{1+\gb} \ ,
  \quad \gt \in [0,\gt_0)  \ ,
 \end{align*}
one gets 
\begin{displaymath}
\begin{split}
\|\overline{T(\sigma)^{n-m-1}\cl A}\|&\; \|\cl A^{-1}(T(\sigma)-U(\sigma))\cl A^{-\gb}\|\;
\|\cl A^\gb U(\sigma)^m\|\\
&\le c_1\,Z(\gb)\,\gL_\gb\frac{\gs^{1-\ga}}{m^\gb} + c_2\,Z(\gb)\,\gL_\gb\frac{1}{(n-1 -m)\,m^\gb} \ .
\end{split}
\end{displaymath}
From Lemma \ref{lem:8.2} of Appendix (\S \ref{sec:9}) we obtain inequalities
\begin{displaymath}
\begin{split}
\sum^{n-2}_{m=1} &\|\overline{T(\sigma)^{n-m-1}\cl A}\|\; \|\cl A^{-1}(T(\sigma)-U(\sigma))\cl A^{-\gb}\|\;
\|\cl A^{\gb} U(\sigma)^m\| \\
& \le c_1\,Z(\gb)\,\gL_\gb\,\gs^{1-\ga}\sum^{n-2}_{m=1} \frac{1}{m^\gb} + c_2\,Z(\gb)\,\gL_\gb
\sum^{n-2}_{m=1}\frac{1}{(n-1 -m)\,m^\gb}\\
& \le \frac{c_1\,Z(\gb)\,\gL_\gb}{1-\gb} (n-1)^{1-\gb}\gs^{1-\ga} + \frac{2c_2\,Z(\gb)\,\gL_\gb}{1-\gb}
\frac{1}{(n-1)^\gb} + c_2\,Z(\gb)\,\gL_\gb\frac{\ln(n-1)}{(n-1)^{\gb}}.
\end{split}
\end{displaymath}

Summarising all these ingredients one gets the estimate
\begin{displaymath}
\begin{split}
\|T&(\gs)^n - U(\gs)^n\| \\
\le & \;c_1 c\,\gs^{1-\ga} + \frac{c_2\,c}{n-1} + \frac{c \,\gL_\ga}{(n-1)^\ga}\;\gs^{1-\ga}\\
& + \frac{c_1\,Z(\gb)\,\gL_\gb}{1-\gb} (n-1)^{1-\gb}\gs^{1-\ga} + \frac{2c_2\,Z(\gb)\,\gL_\gb}{1-\gb}
\frac{1}{(n-1)^\gb} + c_2\,Z(\gb)\,\gL_\gb\frac{\ln(n-1)}{(n-1)^{\gb}} \ .
\end{split}
\end{displaymath}
If we set $\gs := \tau/n$, then
\begin{displaymath}
\begin{split}
\|T&(\tau/n)^n - U(\tau/n)^n\| \\
\le & \;\frac{c_1 c\;T^{1-\ga}}{(n-1)^{1-\ga}} + \frac{c_2\,c}{n-1} + \frac{c \,\gL_\ga\,T^{1-\ga}}{(n-1)}\\
& + \frac{c_1\,Z(\gb)\,\gL_\gb\,T^{1-\ga}}{1-\gb}\frac{1}{(n-1)^{\gb-\ga}} + \frac{2c_2\,Z(\gb)\,
\gL_\gb}{1-\gb}\frac{1}{(n-1)^\gb} + c_2\,Z(\gb)\,\gL_\gb\frac{\ln(n-1)}{(n-1)^{\gb}}.
\end{split}
\end{displaymath}
for $\gt \ge 0$ and $n = 2,3,\ldots$. Hence, there exists a constant $C_{\ga,\gb} > 0$ such 
that \eqref{6.22aa} holds.
\end{proof}
\begin{rem}\label{rem:7.9}
\item[\;\; (i)] If in condition (A6) we put $\gb=1$, then for each $\gga \in (\ga,1)$ there exists a constant
$C_{\ga,\gga} > 0$ such that
 \begin{equation}\label{eq:6.24a}
 \| (e^{-\gt B/n}e^{-\gt \cK_0/n})^{n}-e^{-\gt \cK}\|_{\cl B( L^p(\cI,X))} \leq C_{\ga,\gga} \, 
 \frac{1}{n^{\gga-\ga}} \ ,
\end{equation}
for $\gt \ge 0$ and $n = 2,3,\ldots$.

\item[\;\;(ii)] It is worth to note that our result depends only on domains of the operators $A$ and 
$B(t)$.

\item[\;\;(iii)] Until now, error estimates in operator-norm for the Trotter product formula on 
\textit{Banach spaces} for the time-independent case is proven only under the assumption that at least one of 
the involved operators is generator of a holomorphic contraction semigroup, see \cite{CachZag2001}. 
Therefore, although motivated by \cite{CachZag2001}, the Theorem \ref{TheoremEstimate} is the first result, where 
this assumption is dropped.

\item[\;\;(iv)] If the family of generators is independent of $t \in \cI$, i.e. $B(t) = B$, then 
condition (A6) is automatically satisfied for any $\gb \ge 0$. In particular, we can set $\gb = 1$. Since 
$\cA$ and $\cB$ commute with  $D_0$  we get
\begin{equation}\label{eq:6.24}
(e^{-\gt \cB/n}e^{-\gt \cK_0/n})^n = (e^{-\gt \cB/n}e^{-\gt \cA/n})^n e^{-\gt D_0}, \quad \gt \ge 0, 
\quad n \in \N.
\end{equation}

Now we comment that if one of the operators: $A$ or $B$, is generator of a holomorphic contraction 
semigroup and another one of a contraction semigroup on a Banach space $X$, then from Theorem 3.6 of 
\cite{CachZag2001} we get the existence of constants $b_1 > 0$, $b_2 > 0$  and $\eta > 0$ such that the estimate
\begin{displaymath}
\|(e^{-\gt B/n}e^{-\gt A/n})^n - e^{-\gt C}\|_{\cl B(X)} \le (b_1 + b_2 \gt^{1-\ga})e^{\gt \eta} \ 
\frac{\ln(n)}{n^{1-\ga}} \ , 
\end{displaymath}
holds for $\gt \ge 0$ and $n \in \N$. Applying this result to \eqref{eq:6.24} we immediately obtain the 
existence of a constant $R > 0$ such that
\begin{displaymath}
\|(e^{-\gt \cB/n}e^{-\gt \cK_0/n})^n - e^{-\gt \cK}\|_{\cl B( L^p(\cI,X))} 
\le R \ \frac{\ln(n)}{n^{1-\ga}} \ ,
\end{displaymath}
is valid for $\gt \ge 0$ and $n \in \N$. Note that this estimate is sharper than the estimate in 
\eqref{eq:6.24a}.
\end{rem}

\subsection{Norm convergence for propagators}\label{sec:7.3}

We investigate here the consequences of Theorem \ref{TheoremEstimate} for convergence of the approximants,
$\{U_n(t,s)\}_{(t,s)\in\Delta}$, $n \in \N$,  \eqref{ApproximationPropagatorIntroduction}, to the propagator
$\{U(t,s)\}_{(t,s)\in \gD}$, which solves the non-ACP \eqref{EvolutionProblem}.  

Recall that by (\ref{eq:6.main}) one gets the relation 
\begin{align}\label{eq:7main}
 (\{(e^{-\frac \tau n \cl B}e^{-\frac \tau n \cl K_0})^n-&e^{-\tau (\cl B+\cl K_0)}\}f)(t)
 =\{U_n(t,t-\tau )-U(t,t-\tau )\}\chi_\cI(t-\tau )f(t-\tau )
\end{align}
for $(t,t-\gt) \in \gD$ and $f \in L^p(\cI,X)$, where the Trotter product approximation  
for propagator $U(t,s)$ has the form
\begin{displaymath}
 U_n(t,s):=\stackrel{\longrightarrow}\prod_{j=1}^n e^{-\frac{t-s} n B(s+(n-j+1)\frac{t-s} n)}
 e^{-\frac {t-s} n A},
 \quad (t,s) \in \gD \ .
\end{displaymath}
is increasingly ordered from the left  to the right.
%
\begin{thm}\label{EstimatePropagators}
 Let the assumptions  (A1), (A3), (A4), (A5), and (A6) be satisfied.
If the family of generators $\{B(t)\}_{t\in\cI}$ is $A$-stable and $\gb \in (\ga,1)$,  then there exists 
a constant $C_{\alpha, \beta}>0$ such that
 \begin{equation}\label{eq:6.26}
 \esssup_{(t,s)\in\gD}\|U_n(t,s)-U(t,s)\|_{\cl B(X)} \le \frac{C_{\ga,\gb}}{n^{\gb-\ga}}, 
 \quad n = 2,3,\ldots \ .
\end{equation}
The constant $C_{\ga,\gb}$ coincides with that in the estimate \eqref{6.22aa} of Theorem \ref{TheoremEstimate}.
\end{thm}
%
\begin{proof}
We set
\begin{displaymath}
S_n(t,s) := U_n(t,s) - U(t,s), \quad (t,s) \in \gD, \quad n \in \N \ ,
\end{displaymath}
and
\begin{align*}
 \cl S_n(\tau) := L(\tau)\{(e^{-\frac \tau n \cl B}e^{-\frac \tau n \cl K_0})^n-e^{-\tau
 (\cl B+\cl K_0)}\}: L^p(\cI,X)
 \rightarrow  L^p(\cI,X) \ ,
\end{align*}
for $\gt \ge 0$ and $n =2,3,\ldots \ $. Here $L(\tau)$, $\tau \ge 0$, is the left-shift semigroup 
(\ref{L-shift}). Then by (\ref{eq:7main}) we get 
\begin{align}\label{Sn-Sn}
(\cl S_n(\gt)g)(t) = S_n(t+\gt,t)\chi_\cI(t+\gt)g(t), \quad t \in \cI_0, \quad g \in L^p(\cI,X).
\end{align}
Hence, for any $\tau \in\cI$ and $n \in \N$, the operator $\cl S_n^\tau $ is a multiplication
operator on $ L^p(\cI,X)$ induced by the family $\{S_n(\cdot+\gt,\cdot)\chi_\cI(\cdot+\gt)\}_{\gt\in\cI}$
of bounded operators. Applying first (\ref{L-shift-id}) and then Proposition 
\ref{OperatorNormsBoundedOperator} for (\ref{Sn-Sn}), one gets for $\tau \geq 0$ the equality
\begin{align}\label{main-equal}
 \|&(e^{-\frac \tau n \cl B}e^{-\frac \tau n \cl K_0})^n-e^{-\tau (\cl B+\cl K_0)}\|_{\cl B( L^p(\cI,X))} = \|L(\tau)\{(e^{-\frac \tau n \cl B}e^{-\frac \tau n \cl K_0})^n-e^{-\tau
 (\cl B+\cl K_0)}\}\|_{\cl B( L^p(\cI,X))} \nonumber\\
 &= \|\cl S_n(\tau )\|_{\cl B( L^p(\cI,X))} = \esssup_{t\in\cI_0} \|S_n(t+\gt,t)\chi_\cI(t+\gt)\|_{\cl B(X)} \nonumber \\
 &= \esssup_{t\in\cI_0} \|\{U_n(t+\tau ,t)-U(t+\tau ,t)\}\chi_\cI(t+\tau )\|_{\cl B(X)} = \esssup_{t\in(0,T-\tau ]} \|U_n(t+\tau ,t)-U(t+\tau ,t)\|_{\cl B(X)}.
\end{align}
Now taking into account Theorem \ref{TheoremEstimate} we find
\begin{displaymath}
\esssup_{t\in(0,T-\tau ]} \|U_n(t+\tau,t)-U(t+\tau,t)\|_{\cl B(X)} \le \frac{C_{\ga,\gb}}{n^{\gb-\ga}},
\quad \gt \ge 0,\quad n \in 2,3,\ldots \ ,
\end{displaymath}
which yields \eqref{eq:6.26}.
\end{proof}
\begin{rem}
\item[\;\;(i)] The equality (\ref{main-equal}) shows that estimates  \eqref{6.22aa} and \eqref{eq:6.26} 
are equivalent.

  \item[\;\;(ii)] We note that \textit{a priori} for a fixed $n \in \N$ the operator family 
  $\{U_n(t,s)\}_{(t,s)\in \Delta}$ do not define a propagator since the \textit{co-cycle equation} is, 
  in general, not satisfied. But one can check that
 \begin{align*}
  U_n(t,s)=U_{n-k}\left(t,s+\frac{k}{n}(t-s)\right)U_k\left(s+\frac{k}{n}(t-s), s\right),
 \end{align*}
is satisfied for $0<s\leq t\leq T$, $n\in \N$ and any $k\in\{0,1,\dots, n\}$.
\end{rem}


\section{Example: Diffusion equation perturbed by a time-dependent potential} \label{sec:8}

Here we investigate a non-autonomous problem when diffusion equation is perturbed by a time-dependent 
potential. To this aim consider the Banach space $X=L^q(\Omega)$, where $\Omega\subset\R^d$, $d\geq 2$, 
is a bounded domain with $C^{2}$-boundary  and $q\in (1,\infty)$. Then the equation for the non-ACP 
reads as
\begin{align}
 \dot u(t)=\Delta u(t) -B(t)u(t), ~~u(s)=u_s\in L^q(\Omega), ~~t,s\in \cI_0 \ ,
 \label{EvolProbLaplaceAndTimeDependentPotential}
\end{align}
where $\Delta$ denotes the Laplace operator in $L^q(\Omega)$ with Dirichlet boundary conditions defined by
the mapping
\begin{align*}
 \Delta:\dom(\Delta)=H^2_q(\Omega)\cap \mathring H^1_q(\Omega)\rightarrow L^q(\Omega).
\end{align*}
Then operator $-\Delta$ is generator of a holomorphic contraction semigroup
on $L^q(\Omega)$, \cite[Theorem 7.3.5/6]{Pazy1983}, and $0\in\varrho(A)$. Let $B(t)$
denote a time-dependent scalar-valued multiplication operator given by
\begin{align*}
 (B(t)f)(x)=V(t,x)f(x), ~~\dom(B(t))=\{f\in L^q(\Omega): V(t,x)f(x)\in L^q(\Omega) \} \ ,
\end{align*}
where
\begin{align*}
 V:\cI\times\Omega\rightarrow \C, ~~ V(t,\cdot) \in L^{\varrho}(\Omega) \ .
\end{align*}

For $\alpha \in (0,1)$, the fractional power of $-\Delta$ are defined on the domain 
$\mathring{H}^{2\alpha}_q(\Omega)$ by 
\begin{align*}
(-\Delta)^{\alpha}: \mathring{H}^{2\alpha}_q(\Omega)\rightarrow L^q(\Omega).
\end{align*}
Note, that for $2\alpha<  1/q$, it holds that $\mathring{H}^{2\alpha}_q(\Omega) =  H^{2\alpha}_q(\Omega)$.
The operator $\Delta^*$ is dual to $\Delta$ and it is defined on domain
\begin{equation*}
\dom(\Delta^*)=H^2_{q'}(\Omega)\cap \mathring{H}^1_{q'}(\Omega)\subset L^{q'}(\Omega) \ ,
\end{equation*}
where ${1}/{q}+ {1}/{q'}=1$. Since operators $B(t)$ are scalar-valued, one gets that the dual 
$B(t)^* = \overline{B(t)}: \dom(B(t))\subset L^{q'}(\Omega)\rightarrow L^{q'}(\Omega)$.

\begin{rem}
 Note that the operator $A=-\Delta$ in $L^p(\Omega)$, $p\in(1,\infty)$ with
 Dirichlet boundary conditions verify that \textit{maximal parabolic regularity} condition, see 
 \cite{Arendt2007}.
 In particular this means that
 $\widetilde{\cl K_0} = D_0 +\cl A$ is closed and hence coincides with its closure: 
 $\widetilde{\cl K_0} = \cl K_0$.
\end{rem}

To prove the existence and uniqueness of solution of the non-ACP \eqref{EvolProbLaplaceAndTimeDependentPotential}
and in order to construct the product approximants for this solution, we have to verify the
assumptions (A1) - (A6). In particular, we have to determine the required regularity of 
$V(t,\cdot)\in L^\varrho(\Omega)$ to ensure that (\cite[Corollary 3.7]{CachZag2001})
\begin{equation*}
\dom((-\Delta)^\alpha)\subset\dom(B(t)) \ \ \ {\rm{and}} \ \ \  \dom(\Delta^*)\subset \dom(B(t)^* \ ,
\end{equation*}
or in other words:
\begin{align}\label{OurAimedEmbedding}
 H^{2\alpha}_q(\Omega),H^{2}_{q'}(\Omega)\subset \dom(B(t)) \ .
\end{align}

Using Sobolev embeddings, one obtains general description of the embedding
\begin{align}
 H^s_{\gamma_1} (\Omega) \subset L^{\gamma_2}(\Omega) {\rm ~~for ~}
 \begin{cases}
  \gamma_2\in[\gamma_1, \frac{d \gamma_1/s}{d/s - \gamma_1}], {\rm ~if~} \gamma_1 \in (1, {d}/{s})\\
  \gamma_2\in[\gamma_1, \infty), {\rm ~if~} \gamma_1 \in [{d}/{s}, \infty)
  \end{cases}\label{SobolevEmbeddings} \ .
\end{align}
For our case \eqref{OurAimedEmbedding}, we obtain $
 H^{2\alpha}_q(\Omega)\subset L^r(\Omega) {\rm ~~and~~} H^{2}_{q'}(\Omega)\subset L^\rho(\Omega)$,
for some constants $r,\rho\in(1,\infty]$. Hence, it suffices to ensure
$L^r(\Omega), L^\varrho(\Omega)\subset\dom(B(t))$.
The parameters $r, \rho$ define $\tilde r, \tilde \rho$ via
\begin{align}\label{RelationRTildeRhoTilde} 
 \frac{1}{r}+ \frac{1}{\tilde r} = \frac{1}{q}, ~~\frac{1}{\rho}+
 \frac{1}{\tilde \rho} = \frac{1}{q'} \ , 
\end{align}
and since the operator $B(t)$ is a multiplication operator defined by $V(t,\cdot)$,
the regularity of $V(t,\cdot)$ has to be at least as
\begin{align*}
\varrho = \max\{\tilde r, \tilde \rho\} \ .
\end{align*}
We distinguish these cases collecting them in the Table 1:
\begin{enumerate}
 \item For $d\geq 4$, or for $d=3$ and $\alpha\in(0,\frac{1}{2})$:
 \begin{center}
 \begin{tabular}{l|l|l||l|l|l}
 $q\in$ & $\tilde r\in$ & $\varrho\in$ & $q\in$ & $\tilde r\in$ & $\varrho\in$ \\
\hline
$(1, \frac{d}{d-2\alpha}]$ &$[\frac d {2\alpha},\infty]$ & $(q', \infty]$ &  $(\frac{d}{d-2\alpha},
\frac{d}{d-2}]$&$[\frac d {2\alpha},\infty]$ & $(\frac{d}{2\alpha}, \infty]$ \\
$(\frac{d}{d-2}, \frac{d}{2\alpha})$ &$[\frac d {2\alpha},\infty]$ &$[\frac{d}{2\alpha},
\infty]$ &  $[\frac{d}{2\alpha}, \infty)$ &  $(q,\infty]$&$(q,\infty]$
\end{tabular}
\end{center}

\item For $d=3$ and $\alpha \in [\frac{1}{2}, \frac 3 4)$:
\begin{center}
\begin{tabular}{l|l|l||l|l|l}
$q\in$ & $\tilde r\in$ & $\varrho\in$& $q\in$ & $\tilde r\in$ & $\varrho\in$\\
\hline
$(1, \frac{3}{3-2\alpha}]$&$[\frac{3}{2\alpha}, \infty]$ & $(q', \infty]$ & $(\frac{3}{3-2\alpha}, 2]$ &
$[\frac{3}{2\alpha}, \infty]$ & $(q', \infty]$ \\
$(2, \frac{3}{2\alpha})$ &$[\frac 3 {2\alpha},\infty]$& $[\frac{3}{2\alpha}, \infty]$ &
$[\frac{3}{2\alpha}, \infty)$ &$(q,\infty]$& $(q,\infty]$
\end{tabular}
\end{center}

\item For $d=3$ and $\alpha \in [\frac{3}{4}, 1)$:
\begin{center}
\begin{tabular}{l|l|l||l|l|l}
$q\in$ & $\tilde r \in$ & $\varrho\in$& $q\in$ &$\tilde r \in$& $\varrho\in$\\
\hline
$(1, 2]$ & $[\frac 3 {2\alpha}, \infty]$ & $(q', \infty]$ & $(2, \infty)$ & $(q, \infty]$&  $(q, \infty]$
\end{tabular}
\end{center}

\item For $d=2$ and $\alpha\in(0,\frac 1 2]$:
\begin{center}
\begin{tabular}{l|l|l||l|l|l||l|l|l}
$q\in$ & $\tilde r\in$ & $\varrho\in$& $q\in$ & $\tilde r\in$ & $\varrho\in$ & $q\in$ & $\tilde r\in$ & $\varrho\in$\\
\hline
$(1, 2)$ & $[\frac 1 \alpha, \infty]$ & $[\max\{q', \frac 1 \alpha\}, \infty]$ & $[2,
\frac 1 \alpha)$ & $[\frac 1 \alpha, \infty]$ & $[\frac 1 \alpha, \infty]$ &
$[\frac 1 \alpha, \infty)$ &$(q, \infty]$ & $(q, \infty]$
\end{tabular}
\end{center}

\item For $d=2$ and $\alpha\in(\frac 1 2, 1)$:
\begin{center}
\begin{tabular}{l|l|l||l|l|l||l|l|l}
$q\in$ & $\tilde r\in$ & $\varrho\in$ & $q\in$ & $\tilde r\in$ & $\varrho\in$ & $q\in$ & $\tilde r\in$ & $\varrho\in$\\
\hline
$(1, \frac 1 \alpha)$ & $[\frac 1 \alpha, \infty]$ & $[q', \infty]$ & $[\frac 1 \alpha, 2)$ &
$(q, \infty]$ & $[q', \infty]$  &
$[2, \infty)$ & $(q, \infty]$ & $(q, \infty]$
\end{tabular}
\end{center}
\end{enumerate}
The Existence Theorem \ref{OurEvolutionProblemUniqueSolution} yields the following theorem
%
\begin{thm}\label{ExampleUniqueSolution}
 Let $\Omega\subset \R^d$ be a bounded domain with $C^2$- boundary, let $q\in (1, \infty)$
and let $\alpha\in(0, 1)$. Let $B(t)f=V(t,\cdot)f$ define a scalar valued multiplication operator on
$L^q(\Omega)$ with $V\in L^\infty(\cI, L^{\tilde r}(\Omega))$. Let $\tilde r\in(1,\infty)$ be chosen from
the above tables.
Then, the non-ACP \eqref{EvolProbLaplaceAndTimeDependentPotential} has a unique solution operator (propagator).
\end{thm}
%
\begin{proof}
 Using relation \eqref{RelationRTildeRhoTilde} and  the Sobolev embeddings \eqref{SobolevEmbeddings},
 it is easy to see
 that $\dom((-\Delta)^\alpha)\subset\dom(B(t))$ holds. Hence, the assumptions of Theorem
 \ref{OurEvolutionProblemUniqueSolution} are satisfied.
\end{proof}
\begin{rem}
In \cite{PruessSchnaubelt2001}, the existence of a solution operator for equation
\eqref{EvolProbLaplaceAndTimeDependentPotential} is shown assuming weaker
regularity in space and time for the potential. We assumed uniform boundedness of
$\|B(t)(-\Delta)^\alpha\|_{\cl B(X)}$,
which is indeed too strong, but important for the further considerations.
\end{rem}

Now, we study the convergence of the Trotter product approximants of the solution operator.
We assume that the real part of potential $V(t,x)$ is positive:
\begin{align*}
 {\rm Re}(V(t,x)) \geq 0, {\rm ~~for~a.e.~} (t,x)\in \cI\times\Omega \ .
\end{align*}
Then, for any $t\in \cI$ the operator $V(t,x)$ is a generator of a contraction semigroup on
$X=L^q(\Omega)$ (\cite[Theorem I.4.11-12]{EngNag2000}). Moreover, assumption (A4) is satisfied. 

Now let
\begin{align*}
 F(t):=(-\Delta)^{-1}B(t)(-\Delta)^{-\alpha} : L^q(\Omega)\rightarrow
 H^{2}_q(\Omega)\cap \mathring H^1_q(\Omega)\subset L^q(\Omega).
\end{align*}
Assuming $V\in L^\infty(\cI,L^\varrho(\Omega))$, where $\varrho$ is chosen from the Table 1, we find that
$\dom((-\Delta)^\alpha)\subset\dom(B(t))$ and $\dom(\Delta^*)\subset\dom(B(t)^*)$.
Hence, the operator $F(t)$ is uniformly bounded. It rests only to prove the H\"older-continuity 
of the operator-valued function $t\mapsto F(t)$.

Let $f\in L^q(\Omega)$ and $g\in L^{q'}(\Omega)$. Define $\tilde f = \Delta^{-\alpha} f \in
\mathring{H}^{2\alpha}_q(\Omega)\subset L^r(\Omega)$
and $\tilde g = (\Delta^{-1})^*g=(\Delta^*)^{-1}g\in H^2_{q'}(\Omega)\cap \mathring H^1_{q'}
(\Omega)\subset L^{\rho}(\Omega)$. Then, we get for $t\in \cI$
\begin{align*}
 \la F(t) f, g\ra &=  \la (-\Delta)^{-1}B(t)(-\Delta)^{-\alpha}f, g\ra =\la (-\Delta)^{-\alpha}f, B(t)^* (-\Delta^*)^{-1}g\ra = \la \tilde f, B(t)^* \tilde g\ra \ .
\end{align*}
The boundedness of $\la \tilde f, B(t)^* \tilde g\ra$
can be ensured by $V(t,\cdot)\in L^\tau(\Omega)$, where $\tau\in(1,\infty)$ is defined via
\begin{align}
 \frac{1}{r}+\frac{1}{\tau}+\frac{1}{\rho}=1\label{RelationTau}.
\end{align}
%
%

The following Table 2 for parameters turns out:
\begin{enumerate}
 \item For $d\geq 4$, or for $d=3$ and $\alpha\in(0,\frac{1}{2})$:
 \begin{center}
 \begin{tabular}{l|l||l|l||l|l}
 $q\in$ & $\tau \in$ & $q\in$ & $\tau \in$ & $q\in$ & $\tau \in$\\
\hline
$(1, \frac{d}{d-2}]$ &$(\frac{\frac d {2\alpha}q}{\frac{d}{2\alpha}(q-1) + q},\infty]$ & $[\frac{d}{d-2},
\frac{d}{2\alpha}]$ &  $[\frac{d}{2\alpha+2},\infty] $ & $(\frac d {2\alpha},\infty]$ & $[\frac{\frac d 2 q'}
{\frac d 2(q'-1) +q'}, \infty]$
\end{tabular}
\end{center}

\item For $d=3$ and $\alpha \in [\frac 1 2, 1)$:
\begin{center}
\begin{tabular}{l|l||l|l||l|l}
$q\in$ & $\tau \in$ & $q\in$ & $\tau \in$ & $q\in$ & $\tau \in$ \\
\hline
$(1, \frac 3 {2\alpha}]$ & $(\frac{\frac 3 {2\alpha}q}{\frac 3 {2\alpha}(q-1) +q },
\infty]$ & $[\frac 3{2\alpha}, 3]$ & $(1, \infty]$ & $(3, \infty]$&  $[\frac{\frac 3 2 q'}
{\frac 3 2(q'-1) +q'}, \infty]$
\end{tabular}
\end{center}

\item For $d=2$ and $\alpha\in(0,1)$:
\begin{center}
\begin{tabular}{l||l}
$q\in$ & $\tau\in$\\
\hline
$(1, \infty)$ & $(1, \infty]$
\end{tabular}
\end{center}
\end{enumerate}

Since $r\geq q$, it holds that $\tau\leq \tilde \rho$ and hence, $\tau \leq \varrho = \max\{\tilde r,
\tilde \rho\}$ . We are now able to estimate the product approximation of the non-ACP solution of 
\eqref{EvolProbLaplaceAndTimeDependentPotential}.
%
\begin{thm}
Let $\Omega\subset \R^d$ be a bounded domain with $C^2$- boundary,
let $q\in (1, \infty)$, $\alpha\in(0, 1)$ and $\beta\in(\alpha,1)$. Choose $\varrho, \tau\in(1,\infty)$
from the above Table 1 and 2. Let $B(t)f=V(t,\cdot)f$ define a scalar-valued multiplication operator 
in $L^q(\Omega)$ with
\begin{align*}
 V\in L^\infty(\cI, L^{\varrho}(\Omega))\cap C^\beta(\cI, L^\tau(\Omega)).
\end{align*}
Moreover, let ${\rm Re}(V(t,x))\geq0$ for $t\in \cI$ and for a.e. $x\in \Omega$.

Then, the solution operator $U(t,s)$ of \eqref{EvolProbLaplaceAndTimeDependentPotential} can be
approximated in the operator-norm topology with the error bound
 \begin{align*}
  {\rm sup}_{(t,s)\in\Delta}\|U_n(t,s)-U(t,s)\|_{\cl B(L^q(\Omega))} = O(n^{-(\beta-\alpha)}) \ ,
 \end{align*}
by the Trotter product approximants:
 \begin{align}\label{ApproximationPropagatorExample}
  U_n(t,s) = \stackrel{\longrightarrow} \prod_{j=1}^n e^{-\frac{t-s}{n}V(\frac{n-j+1}{n}t+
  \frac{j-1}{n}s,\cdot)}e^{\frac{t-s}{n}\Delta} \ .
 \end{align}
\end{thm}
%
\begin{proof}
One can easily verify that the operator $(-\Delta)^{-1}B(t)(-\Delta)^{-\alpha}$ is bounded. The stability
condition is satisfied, since we deal with generators of contraction semigroups. The $\esssup$ becomes indeed
a $\sup$ since we have continuous time-dependence for propagator's approximants. Then the claim follows
by virtue of Theorem \ref{EstimatePropagators}.
\end{proof}
We conclude this section by some number of remarks.
%
\begin{rem}~
 \item[\;\;(i)] We focused on domains, which are compact and have $C^2$-boundaries. Our arguments can be
 extended to a more general domains.

 \item[\;\;(ii)] Although the propagator approximants $\{U_n(t,s)\}_{(t,s)\in\Delta}$ defined in
\eqref{ApproximationPropagatorExample} looks elaborate, they have a simple structure. The semigroup 
in $L^q(\R^d)$ generated by the Laplace operator is given by the Gauss-Weierstrass semigroup 
(see for example \cite[Chapter 2.13]{EngNag2000}) defined via
\begin{align*}
 (e^{t\Delta}u)(x) = (T(t)u)(x) = (4\pi t)^{-d/2}\int_{\R^d} dy \, e^{-\frac{|x-y|^2}{4t}}u(y) \ .
\end{align*}
The factors $e^{-\tau V(t_j)}$, $j = 1,2, \ldots, n$,  in the product approximant 
\eqref{ApproximationPropagatorExample} are scalar valued and can be easily computed.

\item[\;\;(iii)] In \cite{Batkai2011}, see Theorem 5.2 , the authors proved for the same approximation (called
there the sequential splitting procedure) a \textit{vector-dependent} convergence rate on a subspace in 
$L^q(\R^d)$, where the potential $V$ is bounded and its commutator with Laplacian verifies a supplementary
\textit{commutator} condition.
\end{rem}


\section{Appendix}\label{sec:9}

The next Gronwall-type lemma is useful. It can be proved by iterating the Volterra-integral equation  
\cite[Theorem 2.25]{Hunter2001}.
%
\begin{lem}\label{GronwallLemma}
 Let $F$ be a real function satisfying
 \begin{align*}
 0\leq F(t) \leq c_1 t^{-\alpha}+c_2\int_0^t F(s) (t-s)^{-\alpha} ds, ~~t>0 \ ,
\end{align*}
for some positive constants $c_1, c_2>0$ and $\alpha\in(0,1)$.
Then there is a constant $C=2c_1$ and a time value $t_0=\sigma_\alpha\cdot\left\{{1}/{c_2} ,
\left({1}/{c_2}\right)^{1/(1-\alpha)}\right\}$ (where $\sigma_\alpha$ depends only on $\alpha\in(0,1)$)
such that $F(t)t^{\alpha}\leq C$  for small $t\in(0,t_0)$.
\end{lem}
Further we prove the following lemma.
%
\begin{lem}\label{lem:8.2}
Let $\gb \in [0,1)$. Then the estimates
\begin{equation}\label{eq:9.1}
\sum^{n-1}_{m=1} \frac{1}{m^\gb} \le \frac{n^{1-\gb}}{1-\gb} \ ,
\end{equation}
and
\begin{equation}\label{eq:9.2}
\sum^{n-1}_{m=1}\frac{1}{(n-m) m^\gb} \le \frac{2}{1-\gb}\frac{1}{n^\gb} + \frac{\ln(n)}{n^\gb} \ ,
\end{equation}
are valid for $n = 2,3,\ldots$ \ .
\end{lem}
%
\begin{proof}
The function $f(x) = x^{-\gb}$, $x > 0$, is decreasing. Hence
\begin{displaymath}
\sum^{n-1}_{m=1} \frac{1}{m^\gb} \le \int^{n-1}_0  dx \, \frac{1}{x^\gb}  \le \frac{(n-1)^{1-\gb}}{1-\gb} \le
\frac{n^{1-\gb}}{1-\gb} \ ,
\end{displaymath}
for $n =2,3,\ldots$\, which proves (\ref{eq:9.1}).

Further, we have
\begin{displaymath}
\sum^{n-1}_{m=1}\frac{1}{(n-m) m^\gb} = \frac{1}{n}\sum^{n-1}_{m=1}\frac{n}{(n-m)m^\gb} =
\frac{1}{n}\sum^{n-1}_{m=1}\frac{1}{m^\gb} + \frac{1}{n}\sum^{n-1}_{m=1}\frac{m^{1-\gb}}{n-m} \ .
\end{displaymath}
Since $\sum^{n-1}_{m=1} {1}/{m^\gb} \le n^{1-\gb}/({1-\gb})$ we get the estimate
\begin{equation}\label{eq:9.3}
\sum^{n-1}_{m=1}\frac{1}{(n-m) m^\gb} \le \frac{1}{1-\gb}\frac{1}{n^\gb} + \frac{1}{n}\sum^{n-1}_{m=1}
\frac{m^{1-\gb}}{n-m} \ .
\end{equation}
Note that the function $f(x) := {x^{1-\gb}}/{(n-x)}$, $x \in [0,n)$, is increasing. Hence for $n\geq 2$,
\begin{displaymath}
\begin{split}
\sum^{n-2}_{m=1}\frac{m^{1-\gb}}{n-m} &\le \int^{n-1}_1 dx \, \frac{x^{1-\gb}}{n-x} 
\le n^{1-\gb}\int^1_{\tfrac{1}{n}} ds \, \frac{(1-s)^{1-\gb}}{s} \le n^{1-\gb}\int^1_{\tfrac{1}{n}} ds \, \frac{1}{s} =  n^{1-\gb}\ln(n).
\end{split}
\end{displaymath}
Consequently we obtain the estimate  
\begin{displaymath}
\frac{1}{n}\sum^{n-1}_{m=1}\frac{m^{1-\gb}}{n-m} = 
\frac{1}{n}\sum^{n-2}_{m=1}\frac{m^{1-\gb}}{n-m} + \frac{1}{n}(n-1)^{1-\gb} \le 
\frac{\ln(n)}{n^{\gb}} + \frac{1}{1-\beta}\frac{1}{n^{\gb}}, \quad n = 2,3,\ldots\, ,
\end{displaymath}
which together with (\ref{eq:9.3}) proves (\ref{eq:9.2}).
\end{proof}







\begin{thebibliography}{10}

\bibitem{Acquistapace1987}
P. Aquistapace and B. Terreni.
\newblock A unified approach to abstract linear nonautonomous parabolic equations.
\newblock {\em Rendiconti del Seminario Matematico della Universita di Padova}, tome 78, 47-107, 1987.

\bibitem{Arendt2007}
W.~Arendt, R.~Chill, S.~Fornaro, and C.~Poupaud.
\newblock {$L^p$}-maximal regularity for non-autonomous evolution equations.
\newblock {\em J. Differential Equations}, 237(1):1--26, 2007.


\bibitem{Batkai2011}
A. B\'atkai, P. Csom\'os, B. Farkas, and G. Nickel.
\newblock Operator splitting for non-autonomous evolution equations.
\newblock {\em Journal of Functional Analysis}, 260: 2163--2190, 2011.

\bibitem{Batkai2012}
A. B\'atkai and E. Sikolya.
\newblock The norm convergence of a Magnus expansion method.
\newblock {\em Cent. Eur. J. Math}, 10(1): 150-158, 2012.

\bibitem{CachZag2001}
V.~Cachia and V.~A. Zagrebnov. 
Operator-norm convergence of the Trotter product formula for holomorphic semigroups.
\newblock {\em J. Operator Theory}, 46 : 199--213, 2001.


\bibitem{CembranosMendoza1997}
P.~Cembranos and J.~Mendoza.
\newblock {\em Banach spaces of vector-valued functions}, volume 1676 of {\em
  Lecture Notes in Mathematics}.
\newblock Springer-Verlag, Berlin, 1997.

\bibitem{EngNag2000}
K.-J. Engel and R.~Nagel.
\newblock {\em One-parameter semigroups for linear evolution equations}.
\newblock Springer-Verlag, New York, 2000.

\bibitem{Evans1976}
D.~E. Evans.
\newblock Time dependent perturbations and scattering of strongly continuous
  groups on {B}anach spaces.
\newblock {\em Math. Ann.}, 221(3):275--290, 1976.

\bibitem{Hunter2001}
J.~K. Hunter and B.~Nachtergaele.
\newblock {\em Applied Analysis}. 
\newblock World Scientific Publishing Co., Singapore, 2001.

\bibitem{IchinoseTamura1998}
T.~Ichinose and H.~Tamura.
\newblock Error estimate in operator norm of exponential product formulas for
  propagators of parabolic evolution equations.
\newblock {\em Osaka J. Math.}, 35(4):751--770, 1998.

\bibitem{Kato1953}
T.~Kato.
\newblock Integration of the equation of evolution in a {B}anach space.
\newblock {\em J. Math. Soc. Japan}, 5:208--234, 1953.

\bibitem{Kato1970}
T.~Kato.
\newblock Linear evolution equations of ``hyperbolic'' type.
\newblock {\em J. Fac. Sci. Univ. Tokyo Sect. I}, 17:241--258, 1970.

\bibitem{Kato1973}
T.~Kato.
\newblock Linear evolution equations of ``hyperbolic'' type. {II}.
\newblock {\em J. Math. Soc. Japan}, 25:648--666, 1973.

\bibitem{Kato1980}
T.~Kato.
\newblock {\em Perturbation theory for linear operators}.
\newblock Classics in Mathematics. Springer-Verlag, Berlin, 1995.

\bibitem{Nei1981}
H.~Neidhardt.
\newblock {On abstract linear evolution equations. {I}.}
\newblock {\em Math. Nachr.}, 103:283--298, 1981.

\bibitem{NeiZag1998}
H.~Neidhardt and V.~A. Zagrebnov.
\newblock On Error Estimates for the Trotter–Kato Product Formula
\newblock {\em Letters in Mathematical Physics}, 44: 169–-186, 1998.

\bibitem{NeiZag2009}
H.~Neidhardt and V.~A. Zagrebnov.
\newblock Linear non-autonomous {C}auchy problems and evolution semigroups.
\newblock {\em Adv. Differential Equations}, 14(3-4):289--340, 2009.

\bibitem{Nickel1996}
G. Nickel.
\newblock On evolution semigroups and nonautonomous {C}auchy problems.
\newblock {\em Diss. Summ. Math.}, 1(1-2):195--202, 1996.

\bibitem{Pazy1983}
A.~Pazy.
\newblock {\em Semigroups of linear operators and applications to partial
  differential equations}.
\newblock Springer-Verlag, New York, 1983.

\bibitem{Phillips1953}
R.~S. Phillips.
\newblock Perturbation theory for semi-groups of linear operators.
\newblock {\em Trans. Amer. Math. Soc.}, 74:199--221, 1953.

\bibitem{Prato1984}
G. da Prato and P. Grisvard.
\newblock Maximal Regularity for Evolution Equations by Interpolation and Extrapolation
\newblock {\em Journal of Functional Analysis}, 58: 107-124, 1984.

\bibitem{PruessSchnaubelt2001}
J.~Pr{\"u}ss and R.~Schnaubelt.
\newblock Solvability and maximal regularity of parabolic evolution equations
  with coefficients continuous in time.
\newblock {\em J. Math. Anal. Appl.}, 256(2):405--430, 2001.

\bibitem{Stephan2016}
A.~Stephan.
\newblock {\it On operator-norm estimates for approximations of solutions of
  evolution equations using the Trotter product formula}.
\newblock Master's thesis, Humboldt-Universit\"at zu Berlin, 2016.

\bibitem{Tan1979}
H.~Tanabe.
\newblock {\em Equations of evolution}.
\newblock Pitman (Advanced Publishing Program), Boston, Mass.-London, 1979.

\bibitem{Trotter1959}
H.~F. Trotter.
\newblock On the product of semi-groups of operators.
\newblock {\em Proc. Amer. Math. Soc.}, 10:545--551, 1959.

\bibitem{Voigt1977}
J.~Voigt.
\newblock On the perturbation theory for strongly continuous semigroups.
\newblock {\em Math. Ann.}, 229(2):163--171, 1977.

\bibitem{VWZ2008}
P.-A. Vuillermot, W.F. Wreszinski, and V. A. Zagrebnov.
\newblock A Trotter-Kato product formula for a class of non-autonomous
evolution equations.
\newblock \textit{Nonlinear Analysis}, 69: 1067–-1072, 2008.

\bibitem{VWZ2009}
P.-A. Vuillermot, W.F. Wreszinski, and V. A. Zagrebnov.
\newblock A general Trotter-Kato formula for a class of evolution operators.
\newblock \textit{J. Funct. Anal.} 257: 2246–-2290, 2009.

\bibitem{Zag2003} V.A.~Zagrebnov, \textit{Topics in the Theory of Gibbs Semigroups},
KU Leuven University Press, Leuven 2003.


\end{thebibliography}

\end{document}